\documentclass[a4paper,twoside]{article}
\usepackage{a4}
\usepackage{amssymb}
\usepackage{amsmath}
\usepackage{upref}

\usepackage[active]{srcltx}
\allowdisplaybreaks[2] 

\usepackage[dvipsnames]{color}

\newcount\minutes \newcount\hours
\hours=\time
\divide\hours 60
\minutes=\hours
\multiply\minutes -60
\advance\minutes \time
\newcommand{\klockan}{\the\hours:{\ifnum\minutes<10 0\fi}\the\minutes}
\newcommand{\tid}{\today\ \klockan}
\newcommand{\prtid}{\smash{\raise 10mm \hbox{\LaTeX ed \tid}}}
\renewcommand{\prtid}{}
\makeatletter
\def\sectionmark#1{} 
\def\subsectionmark#1{}
\newcommand{\sectnr}{\ifnum \c@secnumdepth >\z@
                 \thesection.\hskip 1em\relax \fi}
\def\@evenhead{\footnotesize\rm\thepage\hfil\leftmark\hfil\llap{\prtid}}
\def\@oddhead{\footnotesize\rm\rlap{\prtid}\hfil\rightmark\hfil\thepage}
\def\tableofcontents{\section*{Contents} 
 \@starttoc{toc}}
\makeatother
\makeatletter
\def\@biblabel#1{#1.}
\makeatother
\makeatletter
\let\Thebibliography=\thebibliography
\renewcommand{\thebibliography}[1]{\def\@mkboth##1##2{}\Thebibliography{#1}
\addcontentsline{toc}{section}{References}
\frenchspacing 
\setlength{\@topsep}{0pt}
\setlength{\itemsep}{0pt}%
\setlength{\parskip}{0pt plus 2pt}%
}
\makeatother
\makeatletter
\def\mdots@{\mathinner.\nonscript\!.%
 \ifx\next,.\else\ifx\next;.\else\ifx\next..\else
 \nonscript\!\mathinner.\fi\fi\fi}
\let\ldots\mdots@
\makeatother
\makeatletter
\let\Enumerate=\enumerate
\renewcommand{\enumerate}{\Enumerate%
\setlength{\@topsep}{0pt}
\setlength{\itemsep}{0pt}%
\setlength{\parskip}{0pt plus 1pt}%
\renewcommand{\theenumi}{\textup{(\alph{enumi})}}%
\renewcommand{\labelenumi}{\theenumi}%
}
\let\endEnumerate=\endenumerate
\renewcommand{\endenumerate}{\endEnumerate\unskip}
\makeatother
\makeatletter
\def\@seccntformat#1{\csname the#1\endcsname.\quad}
\makeatother
\newcommand{\authortitle}[2]{\author{#1}\title{#2}\markboth{#1}{#2}}
\newcommand{\auth}[2]{{#1, #2.}}

\newcommand{\art}[6]{{\sc #1, \rm #2, \it #3\/ \bf #4 \rm (#5), \mbox{#6}.}}

\newcommand{\artprep}[3]{{\sc #1, \rm #2, \rm #3.}}
\newcommand{\artin}[3]{{\sc #1, \rm #2,  in #3.}}

\newcommand{\book}[3]{{\sc #1, \it #2, \rm #3.}}
\newcommand{\AND}{{\rm and }}
\def\cprime{{\mathsurround0pt$'$}}
\RequirePackage{amsthm}
\newtheoremstyle{descriptive}%
  {\topsep}   
  {\topsep}   
  {\rmfamily} 
  {}          
  {\bfseries} 
  {.}         
  { }         
  {}          
\newtheoremstyle{propositional}%
  {\topsep}   
  {\topsep}   
  {\itshape}  
  {}          
  {\bfseries} 
  {.}         
  { }         
  {}          
\theoremstyle{propositional}
\newtheorem{thm}{Theorem}[section]
\newtheorem{prop}[thm]{Proposition}
\newtheorem{lem}[thm]{Lemma}
\newtheorem{cor}[thm]{Corollary}
\theoremstyle{descriptive}
\newtheorem{deff}[thm]{Definition}
\newtheorem{example}[thm]{Example}
\newtheorem{remark}[thm]{Remark}

\makeatletter
\renewenvironment{proof}[1][\proofname]{\par
  \pushQED{\qed}%
  \normalfont
  \trivlist
  \item[\hskip\labelsep
        \itshape
    #1\@addpunct{.}]\ignorespaces
}{%
  \popQED\endtrivlist\@endpefalse
}
\makeatother
\newcommand{\setm}{\setminus}
\renewcommand{\emptyset}{\varnothing}
\def\vint_#1{\mathchoice%
          {\mathop{\kern 0.2em\vrule width 0.6em height 0.69678ex depth -0.58065ex
                  \kern -0.8em \intop}\nolimits_{\kern -0.4em#1}}%
          {\mathop{\kern 0.1em\vrule width 0.5em height 0.69678ex depth -0.60387ex
                  \kern -0.6em \intop}\nolimits_{#1}}%
          {\mathop{\kern 0.1em\vrule width 0.5em height 0.69678ex depth -0.60387ex
                  \kern -0.6em \intop}\nolimits_{#1}}%
          {\mathop{\kern 0.1em\vrule width 0.5em height 0.69678ex depth -0.60387ex
                  \kern -0.6em \intop}\nolimits_{#1}}}
\newcommand{\longemb}{\lhook\joinrel\longrightarrow}

\newcommand{\embed}{\hookrightarrow}
\newcommand{\Cp}{{C_p}}

\DeclareMathOperator{\diam}{diam}
\DeclareMathOperator{\supp}{supp}

\DeclareMathOperator{\dist}{dist}

\DeclareMathOperator*{\esssup}{ess\,sup}
\newcommand{\bdry}{\partial}
\newcommand{\bdy}{\bdry}
\newcommand{\grad}{\nabla}
\newcommand{\loc}{_{\rm loc}}
{\catcode`p =12 \catcode`t =12 \gdef\eeaa#1pt{#1}}      
\def\accentadjtext#1{\setbox0\hbox{$#1$}\kern   
                \expandafter\eeaa\the\fontdimen1\textfont1 \ht0 }
\def\accentadjscript#1{\setbox0\hbox{$#1$}\kern 
                \expandafter\eeaa\the\fontdimen1\scriptfont1 \ht0 }
\def\accentadjscriptscript#1{\setbox0\hbox{$#1$}\kern   
                \expandafter\eeaa\the\fontdimen1\scriptscriptfont1 \ht0 }
\def\accentadjtextback#1{\setbox0\hbox{$#1$}\kern       
                -\expandafter\eeaa\the\fontdimen1\textfont1 \ht0 }
\def\accentadjscriptback#1{\setbox0\hbox{$#1$}\kern     
                -\expandafter\eeaa\the\fontdimen1\scriptfont1 \ht0 }
\def\accentadjscriptscriptback#1{\setbox0\hbox{$#1$}\kern 
                -\expandafter\eeaa\the\fontdimen1\scriptscriptfont1 \ht0 }
\def\itoverline#1{{\mathsurround0pt\mathchoice
        {\rlap{$\accentadjtext{\displaystyle #1}
                \accentadjtext{\vrule height1.593pt}
                \overline{\phantom{\displaystyle #1}
                \accentadjtextback{\displaystyle #1}}$}{#1}}
        {\rlap{$\accentadjtext{\textstyle #1}
                \accentadjtext{\vrule height1.593pt}
                \overline{\phantom{\textstyle #1}
                \accentadjtextback{\textstyle #1}}$}{#1}}
        {\rlap{$\accentadjscript{\scriptstyle #1}
                \accentadjscript{\vrule height1.593pt}
                \overline{\phantom{\scriptstyle #1}
                \accentadjscriptback{\scriptstyle #1}}$}{#1}}
        {\rlap{$\accentadjscriptscript{\scriptscriptstyle #1}
                \accentadjscriptscript{\vrule height1.593pt}
                \overline{\phantom{\scriptscriptstyle #1}
                \accentadjscriptscriptback{\scriptscriptstyle #1}}$}{#1}}}}
\newcommand{\limminus}{{\mathchoice{\raise.17ex\hbox{$\scriptstyle -$}}
                {\raise.17ex\hbox{$\scriptstyle -$}}
                {\raise.1ex\hbox{$\scriptscriptstyle -$}}
                {\scriptscriptstyle -}}}
\newcommand{\limplus}{{\mathchoice{\raise.17ex\hbox{$\scriptstyle +$}}
                {\raise.17ex\hbox{$\scriptstyle +$}}
                {\raise.1ex\hbox{$\scriptscriptstyle +$}}
                {\scriptscriptstyle +}}}
\newcommand{\Om}{\Omega}
\newcommand{\Omkj}{\Om_{k,j}}
\newcommand{\clOm}{\overline{\Om}}
\renewcommand{\phi}{\varphi}
\newcommand{\al}{\alpha}
\newcommand{\be}{\beta}
\newcommand{\eps}{\varepsilon}
\newcommand{\la}{\lambda}
\newcommand{\II}{{\cal I}}
\newcommand{\La}{\Lambda}
\newcommand{\de}{\delta}\newcommand{\ga}{\gamma}
\newcommand{\sig}{\sigma}
\newcommand{\Y}{\mathcal{Y}}

\newcommand{\p}{{$p\mspace{1mu}$}}
\newcommand{\R}{\mathbf{R}}
\newcommand{\N}{\mathbf{N}}

\newcommand{\dmu}{d\mu}
\renewcommand{\kappa}{\varkappa}

\newcommand{\Np}{N^{1,p}}
\newcommand{\Wp}{W^{1,p}}
\newcommand{\Mp}{M^{1,p}}

\newcommand{\Dp}{D^p}

\newcommand{\Ct}{{\widetilde{C}}}

\newcommand{\ub}{\bar{u}}

\newcommand{\qb}{\bar{q}}

\numberwithin{equation}{section}
\newenvironment{ack}{\medskip{\it Acknowledgement.}}{}

\newcommand{\Rn}{\mathbf{R}^n}

\begin{document}

\authortitle{Jana Bj\"orn \AND
Agnieszka Ka\l{}amajska} 
{Poincar\'e inequalities and compact embeddings
into weighted $L^q$ spaces}
\title{Poincar\'e inequalities and compact embeddings \\
from Sobolev type spaces \\
into weighted $L^q$ spaces on metric spaces}
\author{Jana Bj\"orn \\
\it\small Department of Mathematics, Link\"oping University, 
\it\small SE-581 83 Link\"oping, Sweden\/{\rm ;}\\
\it \small jana.bjorn@liu.se, ORCID\/\textup{:} 0000-0002-1238-6751
\\
\\
Agnieszka Ka\l amajska \\
\it\small Institute of Mathematics,
University of Warsaw, \\
\it\small ul.\ Banacha 2, PL-02097 Warszawa, Poland\/{\rm ;} \\
\it \small kalamajs@mimuw.edu.pl, ORCID\/\textup{:} 0000-0001-5674-8059
}

\date{}
\maketitle

\noindent{\small
{\bf Abstract}.
We study compactness and boundedness of embeddings from Sobolev type
spaces on metric spaces into $L^q$ spaces with respect to another measure. 
The considered Sobolev spaces can be of fractional order and 
some statements allow also nondoubling measures.
Our results are formulated in a general form,
using sequences of   covering families 
and local Poincar\'e type inequalities.
We show how to construct such suitable coverings and Poincar\'e inequalities.
For locally doubling measures, we prove a self-improvement property for two-weighted Poincar\'e inequalities, 
which applies also to lower-dimensional measures.

We simultaneously treat various Sobolev spaces,
such as the Newtonian, fractional Haj\l asz and Poincar\'e type spaces,
for rather general measures and sets, including fractals and domains with fractal boundaries. 
By considering lower-dimensional measures on the boundaries of such domains,
we obtain trace embeddings for the above spaces.
In the case of Newtonian spaces we exactly characterize 
 when embeddings into 
$L^q$ spaces with respect to another measure are compact.
Our tools are illustrated by concrete examples.
For measures satisfying suitable dimension conditions, 
we recover several classical embedding theorems 
on domains and fractal sets in $\R^n$.
}



\smallskip
\noindent
{\small Declarations of interest: None.}

\section{Introduction}

The classical Rellich--Kondrachov compactness theorem says that if
$\Om\subset\R^n$ is a bounded domain with a sufficiently smooth boundary
and $1\le p<n$, then the embedding
$W^{1,p}(\Om) \embed L^q(\Om)$ is compact for all $1\le q<np/(n-p)$,
see e.g.\ Ziemer~\cite[Theorem~2.5.1 and Exercise~2.3]{Ziem}.
For $p\ge n$ this holds for all $q\ge1$.

In Haj\l asz--Koskela~\cite[Section~8]{HaKobook}, similar compactness results were
proved in the setting of metric measure spaces 
which in most typical situations support global Poincar\'e inequalities
and are either geometrically
doubling or equipped with doubling measures.

In this paper we study compactness and boundedness 
of embeddings from subsets $\Y$ of Sobolev type  spaces, 
defined on a metric measure space $X=(X,d,\mu)$, into $L^q$ spaces 
on a measurable totally bounded
set $E\subset X$ and with respect to possibly another measure~$\nu$.
Our goal is  to formulate and prove the  results under least possible
assumptions, which include various earlier results as special cases. 
We are able to simultaneously treat various spaces of Sobolev 
type, including the Newtonian $\Np$, fractional Haj\l asz $M^{\al,p}$
  and Poincar\'e $P^{\al,p}_\tau$ spaces,
for rather general measures and in metric spaces, including fractals.
Our results can be extended to Besov spaces, but we omit such treatment here.

Our main idea for compactness, formulated in an abstract form in
Assumptions~\ref{ass-comp} and Theorem~\ref{thm-precompact},
is to construct countably many 
local Poincar\'e inequalities involving the measures $\mu$ and $\nu$,
and satisfied on certain covering sets $E_i= E_i(m)$ 
of $E$ and suitable $E_i^{'}=E_i^{'}(m)\subset X$, for each $m=1,2,\ldots$:
 \begin{equation}  \label{eq-intro-PI-gen}
   \biggl(\int_{E_i} |u-{a_{E_i}(u)}|^q \,d \nu \biggr)^{1/q}
        \le C_m  \biggl( \int_{E'_i} g^p \,\dmu \biggr)^{1/p},
\end{equation}
whenever $u:X\to[-\infty,\infty]$ belongs to the appropriate function
space.
As a byproduct, we also obtain a sufficient 
condition for the boundedness  of  the above embeddings,
in which case it suffices to have one covering net, see
Assumptions~\ref{ass-bdd} and Theorem~\ref{thm-bounded}.
We consider mainly the case when $1\le p\le q<\infty$.

The generalized ``gradient'' $g$ in \eqref{eq-intro-PI-gen}
serves as a substitute for $|\grad u|$ 
in the classical Poincar\'e inequalities in the Euclidean spaces and
on  manifolds.
For a fixed $m$, the covering sets $E_i$ and $E_i^{'}$ 
form a finite covering 
family of the sets for which the embedding is considered and, 
intuitively, shrink as $m\to \infty$. 
They will often be constructed from balls with radii $r_m\to0$, but
other choices are possible.

The most important property of  such a covering is that we can  control 
simultaneously the constants $C_m$ in the Poincar\'e  
inequalities~\eqref{eq-intro-PI-gen}, as well as the overlap $N_m$ 
of the covering family $\{E'_i\}$ associated with each $m$. 
This is enforced by the requirement that $C_mN_m^{1/p}\to 0$ as $m\to\infty$. 
A concrete application of this approach is
in Theorem~\ref{thm-w-al-v-be}, where we use integrability
assumptions on the weights $w$ and $v$ in $\R^n$ to obtain an explicit
sufficient condition for the compactness of the embedding 
\[
W^{1,p}(\R^n,w\,dx)\embed L^q(E,v\,dx)
\] 
 for bounded measurable sets $E\subset \R^n$.
Its sharpness for the limiting exponent is illustrated in Example~\ref{ex-optimal-emb}, where
the embedding $W^{1,2}(\R^2,w\,dx)\embed L^2(B(0,1),v\,dx)$ is
compact, while there is no bounded embedding into $L^q(B(0,1),v\,dx)$
for any $q>2$.

Poincar\'e inequalities with suitable exponents $p$ and $q$ are the main ingredient 
in our method. 
In Theorem~\ref{thm-two-weight-PI-eps} we therefore prove the two-weighted
Poincar\'e type inequality \eqref{eq-intro-PI-gen} in spaces with a
good \emph{domain measure} $\mu$ and a rather general \emph{target measure} $\nu$. 
An explicit special case is the following result, proved in 
Section~\ref{sect-derive-PI}.
It applies for example in the setting of weighted $\R^n$, as in 
Heinonen--Kilpel\"ainen--Martio~\cite{HeKiMa},
as well as on many metric measure spaces,
together with a
possibly ``lower-dimensional'' measure $\nu$ defined  on good subsets
 as in \eqref{eq-d-set}, including fractals.
This interpretation seems to be new even in unweighted $\R^n$
and is one of our main contributions.

When $\al=1$ and $\mu=\nu$ satisfies the lower
bound~\eqref{eq-dim-mu-only-s}, we partially recover Theorem~5.1
in Haj\l asz--Koskela~\cite{HaKobook} and Theorem~6   
in Alvarado--G\'orka--Haj\l asz~\cite{AlGoHaj}.
Self-improve\-ments of Poincar\'e inequalities with two measures were 
earlier obtained in e.g.\ 
Chanillo--Wheeden~\cite{ChanWhe85} and Bj\"orn~\cite{FennAnn}
in the setting of $A_p$ weights in~$\R^n$.

\begin{thm}[Self-improvement of Poincar\'e inequalities on subsets]
\label{thm-intro-PI}
Assume that $1\le p<q<\infty$, that $\mu$ and $\nu$ satisfy the doubling condition 
for all balls $B\subset X$ centred 
in a uniformly perfect set $E\subset X$, and that $\mu$ supports the \p-Poincar\'e 
inequality 
\begin{equation}   
\label{eq-def-abstr-PI}
\vint_B |u-u_{B,\mu}| \,d\mu
        \le C[\diam(B)]^\al \biggl( \vint_{\la B} g^p \,d\mu
        \biggr)^{1/p},
\end{equation}
with $\al>0$ and dilation $\la\ge1$,  
for all such balls and a pair of functions $(u,g)$,
where $u$ is assumed to have $\mu$-Lebesgue points $\nu$-a.e.\ in $E$.

Then the Poincar\'e type inequality \eqref{eq-intro-PI-gen} holds for 
the pair $(u,g)$ and all balls 
$B=B(x,r)\subset X$ centred in $E$, with $E_i=E\cap B$, $E'_i=2\la B$, 
$a_{E_i}(u)=\mu(B)^{-1}\int_B u\,d\mu$, 
the exponent $q$ in the left-hand side of \eqref{eq-intro-PI-gen}
replaced by any $q'<q$ and the Poincar\'e constant $C_m$ replaced by 
\begin{equation}  \label{eq-PI-w-Theta-intro}
C(r)= C' \nu(E\cap B)^{1/q'-1/q} \sup_{0<\rho\le r}
\sup_{x\in E\cap B} \frac{\rho^\al \nu(B(x,\rho))^{1/q}}{\mu(B(x,\la \rho))^{1/p}}.
\end{equation}
The constant $C'$ in \eqref{eq-PI-w-Theta-intro} depends only on
  $C$ in \eqref{eq-def-abstr-PI} and on other
  fixed parameters {\rm(}including $q'${\rm)}, but not on $B$, $u$ and $g$.
If the pair $(u,g)$ satisfies the truncation property, then also $q'=q$ is allowed.
\end{thm}

The ``gradient'' $g$ is allowed to be both  nonlinear and nonlocal, for example:
\begin{enumerate}
\item[$\bullet$]
The Haj\l asz gradient coming from the pointwise inequality
\begin{equation}     \label{eq-Haj-grad-intro}
|u(x)-u(y)| \le d(x,y)^\al (g(x)+g(y)) 
\quad \text{for $\mu$-a.e.\ } x,y\in X,
\end{equation}
and defining the fractional Haj\l asz Sobolev space
$M^{\al,p}(X,\mu)$ in Section~\ref{sect-frac-Sob-nonlocal}.
\item[$\bullet$]
The upper gradient on metric spaces,  defined by curve integrals 
\begin{equation*} 
        |u(x) - u(y)| \le \int_{\gamma} g\,ds 
\quad \text{for all $x,y\in X$ and all curves $\ga$ connecting them,}
\end{equation*}
and defining the Newtonian Sobolev space $\Np(X,\mu)$ in Section~\ref{sect-Newt}.
\end{enumerate}
Our next result holds for fractional Haj\l asz spaces with very mild
assumptions on the measure $\mu$ and
generalizes Theorem~2 in Ka\l amajska~\cite{agnieszka},
which dealt with $\alpha =1$.
It is a direct consequence of Proposition~\ref{prop-q=p-M-al}
with $\theta=0$, where
the doubling assumption on $X$ is partially relaxed. 
Note that  $\mu$ need not be doubling.

\begin{prop}[Compactness of $M^{\al,p}$ in $L^p$ with general $\mu$]
\label{prop-M-to-Lp-comp-intro}
Assume that $X$ is 
bounded and doubling as in Definition~\ref{def-doubling}, 
and that $0<\mu(B)<\infty$ for all balls $B$ in $X$.
Then the embedding $M^{\al,p}(X,\mu)\embed L^p(X,\mu)$ is compact for
all $p\ge1$ and $\al>0$.
In particular, if $\mu$ is not a finite sum of atoms, then $M^{\al,p}(X,\mu)\ne L^p(X,\mu)$.
\end{prop}

Two-weighted compactness results for $M^{\al,p}$ and the more general
$E$-restricted Poincar\'e Sobolev spaces $P^{\al,p}_{\tau,E}$, introduced in Definition~\ref{deff-al-poinc-spa},
are provided in Theorem~\ref{thm-comp-for-q'<q-gen} and
Corollary~\ref{cor-comp-dim}. The following
\emph{local doubling and dimension conditions} for $\mu$ and $\nu$
on $E\subset X$ play a major role in these results.
\begin{enumerate}
\renewcommand{\theenumi}{\textup{\bf (\Alph{enumi})}}%
\renewcommand{\labelenumi}{\theenumi}
\setcounter{enumi}{3}
\item
\label{ass-doubl} 
For all $x\in E$  and all $0<r'<r\le r_0$, with $C,C', C'',\de, s,\sig>0$ 
independent of $x$, $r'$ and $r$\/{\rm:}
\end{enumerate}
\vskip -2\medskipamount
\begin{align}
& \setbox0\hbox{$\displaystyle 0<\mu(B(x,2r))$}
\hskip-\wd0
\begin{array}{@{}l@{}}
0<\mu(B(x,2r))\le C\mu(B(x,r)) <\infty, \\
0<\nu(B(x,2r))\le C\nu(B(x,r)) <\infty,
\end{array}  \label{eq-def-doubl-on-E} \\
\frac{\nu({B(x,r')})}{\nu(B(x,r))}
&\le C \Bigl( \frac{r'}{r} \Bigr)^{\de}, \label{eq-dim-est-sig-E}  \\
\mu(B(x,r)) &\ge C' r^s,   \label{eq-dim-mu-only-s}  \\
\nu(B(x,r)) &\le C'' r^\sig.   \label{eq-dim-nu-only-sig} 
\end{align}

In particular, Corollary~\ref{cor-comp-dim} with $E=X$ recovers 
the well-known compactness of the embedding
$M^{\al,p}(X,\mu)\embed L^q(X,\mu)$ for  $q(s-\al p)<sp$,
and generalizes it to two  measures, including lower-dimensional target measures $\nu$. 
When discussing embeddings into $L^q(E,\nu)$, we mean
that $\ub\in L^q(E,\nu)$, where
\begin{equation}  \label{eq-Leb-repr}
\ub(x):= \limsup_{r\to0} \vint_{B(x,r)} u\,d\mu, \quad x\in E
\end{equation}
and the integral average $\vint_{}$ is as in \eqref{eq-average}.

For the Newtonian spaces $\Np(X,\mu)$, defined by
upper gradients, a standard assumption is the following
\p-\emph{Poincar\'e inequality} for  all (or some) balls $B$ and  
all $u\in L^1\loc(X,\mu)$ 
with a minimal \p-weak upper gradient $g_u\in L^p(X,\mu)$,
\begin{equation} \label{PI-ineq}
   \vint_{B} |u-u_{B,\mu}| \,\dmu
        \le C \diam (B) \biggl( \vint_{\lambda B} g_u^p \,\dmu \biggr)^{1/p},
\end{equation}
where  $C>0$ and $\lambda \ge 1$ are independent of  $u$ and $B$, and 
we implicitly assume that $0<\mu(B)<\infty$ for all such balls $B$.
We prove the following characterization of compact embeddings for
$\Np$, which seems new in the context of metric spaces.
It is a direct consequence of
Theorem~\ref{thm-comp-for-q'<q-gen}\,\ref{it-Th-trunc}, together with
Propositions~\ref{prop-comp-imp-M-to0} and~\ref{prop-Th<infty-necess},
and generalizes  a similar condition from $\R^n$ 
due to Maz\cprime ya~\cite[Theorem~8.8.3]{Mazya}.
For the definition of $\Np$ and $\Dp$, see Section~\ref{sect-Newt}.

\begin{thm}[Characterization of embeddings for $\Np$]   
\label{thm-Mazya-cond-intro-new}
Assume that $E\subset X$ is totally bounded and
that for all balls centred in $E$ and of radius at most~$r_0${\rm:}
\begin{enumerate}
\item The doubling conditions~\eqref{eq-def-doubl-on-E} and
  \eqref{eq-dim-est-sig-E} in Assumptions~\ref{ass-doubl} hold for $\mu$ and~$\nu$.
\item The domain measure $\mu$ supports the Poincar\'e inequality~\eqref{PI-ineq} with
  $g_u$.  
\end{enumerate}
Let $\Y = \Np(X,\mu)$ or $\Y = \Dp(X,\mu)\cap L^1(E,\nu)$, 
equipped with the norms
\begin{equation*}   
\|u\|_{\Np(X)} \quad \text{and} \quad
\|g_u\|_{L^p(X,\mu)} + \|u\|_{L^1(E,\nu)}, \quad \text{respectively.}
\end{equation*}
Then the embedding $\Y \embed L^{q}(E,\nu)$ for $q>p$ is
bounded whenever 
\begin{equation}     \label{eq-bdd-cond-intro}
\sup_{0<r<r_0} \sup_{x\in E} \frac{r \nu(B(x,r))^{1/q}}{\mu(B(x,r))^{1/p}} < \infty,
\end{equation}
and compact whenever 
\begin{equation}     \label{eq-Mazya-cond-intro-new}
\sup_{x\in E} \frac{r \nu(B(x,r))^{1/q}}{\mu(B(x,r))^{1/p}} \to 0,
\quad \text{as } r\to0.
\end{equation}
If $\nu$ also satisfies the measure density condition 
\begin{equation}  \label{eq-meas-density}
\nu(B(x,r)\cap E)\ge c\nu(B(x,r)) \quad \text{with some $c>0$,}
\end{equation}
for all $x\in E$ and $0<r\le r_0$ {\rm(}in particular, if
$\nu(X\setm E)=0${\rm)}, then \eqref{eq-bdd-cond-intro}
and~\eqref{eq-Mazya-cond-intro-new} are also necessary for the
boundedness{\rm/}compactness of $\Y \embed L^{q}(E,\nu)$, respectively.
\end{thm}

When $\mu$ and $\nu$ satisfy the dimension conditions
\eqref{eq-dim-mu-only-s} and \eqref{eq-dim-nu-only-sig}, 
Theorem~\ref{thm-Mazya-cond-intro-new} yields the sufficient conditions
$q(s-p)<\sig p$ and $q(s-p)\le\sig p$ for compactness/boundedness of the embedding 
$\Np(X,\mu)\embed L^q(E,\nu)$.

Two particular settings for our results are when the target measure
$\nu$ is a lower-dimensional or codimensional measure, restricted to suitable
subsets of well-behaved metric measure spaces (Example~\ref{ex-d-Hausdorff}), and when the
domain measure $\mu$ is restricted to sufficiently regular domains in 
 such spaces (Theorem~\ref{thm-mu-al-new}). 
By combining these  two approaches, we obtain the following trace embedding
result, proved at the end of Section~\ref{sect-examples}.
These trace embeddings are similar to those in
  Mal\'y~\cite[Proposition~4.16]{MalyBesov}, proved for $\Mp$ and $\Np$ by means of Besov
  spaces. However, the assumptions on the target measure $\nu$ in
  \cite{MalyBesov} are different from ours, while our assumptions on
  the domain measure $\mu$ are somewhat weaker.
See Remark~\ref{rem-when-mu|Om} for typical situations when
  $\mu|_\Om$ satisfies these assumptions.
The obtained exponents for embeddings into $L^q$
are sharp for the dimensions associated with $\mu$ and $\nu$.
The proof shows that $M^{\al,p}(\Om,\mu)$ can be replaced by the
space $P^{\al,p}_{\tau,E}(\Om,\mu)$ from Definition~\ref{deff-al-poinc-spa}.

\begin{prop}[Trace embeddings]   \label{prop-bdry}
Let $\Om\subset X$ be open and $E\subset F\subset\clOm$ such that
the $d$-dimensional Hausdorff measure $\La_d$ 
satisfies the dimension condition 
\begin{equation}  \label{eq-d-set}
C r^d \le \La_d(F\cap B(x,r)) \le C' r^d \quad \text{for all $x\in E$
  and $0<r\le r_0$.}
\end{equation}
Assume that $E$ is totally bounded and $\La_d$-measurable and that for
all balls centred in $E$ and of radius at most~$r_0$, 
the restriction $\mu|_\Om$ satisfies the doubling and dimension conditions~\eqref{eq-def-doubl-on-E} and
\eqref{eq-dim-mu-only-s} in Assumptions~\ref{ass-doubl} with exponent $s$.

If $\al>0$ and $d>s-\al p$, then functions in $M^{\al,p}(\Om,\mu)$ 
have traces on $E$, defined by the integral average
\begin{equation}    \label{eq-extend-limsup}
u(x):= \lim_{r\to0}   \vint_{B(x,r)\cap\Om} u\,d\mu
\quad \text{for $\La_d$-a.e.}~x\in E,
\end{equation}
and the trace mapping $M^{\al,p}(\Om,\mu) \embed L^q(E,\La_d)$
is compact whenever $q(s-\al p)<dp$. 

If $d>s-p$ and the restriction $\mu|_\Om$ also supports the 
\p-Poincar\'e inequality~\eqref{PI-ineq} for $g_u$ on $\Om$, then the
trace mapping $\Np(\Om,\mu) \embed L^q(E,\La_d)$, defined by~\eqref{eq-extend-limsup},
is bounded whenever $q(s-p)\le dp$, and compact if $q(s-p)<dp$.
\end{prop}

The tight connection between Poincar\'e inequalities and compact embeddings was 
in the setting of metric spaces exploited in Haj\l asz--Koskela~\cite[Section~8]{HaKobook}.
Another approach, based on the pointwise inequality~\eqref{eq-Haj-grad-intro} and
suitable coverings by balls with controlled overlap
was used in Ka\l amajska~\cite{agnieszka}.
See also Ivanishko--Krotov~\cite{IvanKrotov},
  Krotov~\cite{krotov} and Romanovski\u{\i}~\cite{Romanovskii} for compact
  embeddings of other abstractly defined Sobolev spaces on metric
  spaces. 
The paper G\'orka--Kostrzewa~\cite{gorkaKos} deals with compact
embeddings for Sobolev spaces defined on metrizable groups.
Haj\l asz--Koskela~\cite{HaKo-imbed} and Haj\l asz--Liu~\cite{hajzu} 
related compactness of Sobolev embeddings to the existence of 
embeddings into better spaces of Orlicz type.

The method proposed in this paper  combines the Poincar\'e inequalities
from~\cite{HaKobook} and the well-controlled coverings from~\cite{agnieszka}.
We formulate our results in the setting of metric
measure spaces but some more general formulations are possible as well.  
Theorems~\ref{thm-precompact}
  and~\ref{thm-bounded} could equally well be formulated in  measure spaces, 
without an underlying metric, or in quasimetric spaces,  
but we omit such generalizations.

We have recently learned about the paper Chua--Rodney--Wheeden~\cite{ChuaRodWhe}, where a
similar abstract treatment, also based on covering families and
a good control of certain generalized gradients, was used to derive compact
embeddings for
degenerate Sobolev spaces 
associated with a nonnegative quadratic form in $\R^n$. 
Despite some similarities, our assumptions on the involved
sets, measures, exponents and inequalities are different
from~\cite{ChuaRodWhe}, see Remark~\ref{rem-Haj-Liu-opt}\,(c).

Our main contributions to the current literature are the following:

$\bullet$
We allow two different measures, both for the
  embeddings 
and the self-im\-pro\-ve\-ment of Poincar\'e inequalities. 
The two-weighted Theorem~\ref{thm-intro-PI}
has recently found applications  in 
\cite{hyptrace} and 
Butler~\cite{ButlerET} to Lebesgue points for Besov spaces on metric spaces.

$\bullet$
By using a lower-dimensional target measure $\nu$,
we can include also trace embeddings 
  (Proposition~\ref{prop-bdry} and Example~\ref{ex-d-Hausdorff}).
  Our approach is direct without the use of Besov spaces on the boundary. 
  The conditions imposed on the measures $\mu$ and $\nu$ are rather flexible
  and natural.

$\bullet$
Some of the measures are not required
  to be doubling but can still support Poincar\'e type inequalities
as in~\eqref{eq-intro-PI-gen} with good enough $C_m$, see Remark~\ref{rem-nondoubl-meas}.
In Theorem~\ref{thm-w-al-v-be}, we consider 
nondoubling $B_p$ weights on $\R^n$
and prove compactness up to, and including, a limiting exponent $q$.
 Proposition~\ref{prop-M-to-Lp-comp-intro} for $M^{\al,p}$ and a similar weaker 
  embedding for the Poincar\'e Sobolev space $P^{\al,p}_{\tau,E}$ 
hold for very general measures, since
  Poincar\'e type inequalities come for free in these cases.

$\bullet$
Even when the doubling condition is used, it is required only
  for small balls centred in $E$. 
This is not the same as assuming
  that the restriction of the measure to $E$ is doubling and cannot
  be treated by considering $E$ as a metric space in its own
  right, see Example~\ref{ex-Whitney-modif}. In particular for the
  trace results, this assumption is much weaker than a global doubling
  condition, see Example~\ref{eq-why-s-on-E}.
Also the Poincar\'e inequality is only required for balls centred
  in $E$, which is captured in Definition~\ref{deff-al-poinc-spa} of the
$E$-restricted Poincar\'e Sobolev space $P^{\al,p}_{\tau,E}$.

$\bullet$
All smoothness exponents $\al>0$ are allowed 
for the fractional Sobolev spaces $M^{\alpha,p}$ and $P^{\al,p}_{\tau,E}$,
as well as for Poincar\'e inequalities~\eqref{eq-def-abstr-PI}.
Most earlier results deal with $\al=1$, while
$\al<1$ can be obtained from $\al=1$ by snowflaking. 
The case of $\al>1$ is less studied and applies in particular to fractals.

$\bullet$
We formulate our sufficient conditions for compactness/boundedness 
using \eqref{eq-bdd-cond-intro} and \eqref{eq-Mazya-cond-intro-new}, rather than 
in terms of dimensions, see Theorem~\ref{thm-comp-for-q'<q-gen}. 
This gives sharper results in general. 
For Newtonian spaces $\Np$ with a good domain measure~$\mu$, it leads to the exact 
characterizion of compact embeddings 
in Theorem~\ref{thm-Mazya-cond-intro-new}.

The paper is organized as follows. 
In Section~\ref{sect-comp} we give the necessary definitions 
and prove our abstract compactness and boundedness results,
Theorems~\ref{thm-precompact} and~\ref{thm-bounded}. 
In Section~\ref{sect-discuss}
we  provide a more thorough discussion of the general assumptions in these theorems, 
such as the choice of covering families and the functionals
$a_{E_i}(u)$ in~\eqref{eq-intro-PI-gen}.
Section~\ref{sect-nondoubl}, contains compactness results with
nondoubling measures.

In Section~\ref{sect-derive-PI} we construct the
local $(q,p)$-Poincar\'e inequalities~\eqref{eq-intro-PI-gen} for locally doubling measures,
starting from the weaker \p-Poincar\'e inequalities~\eqref{eq-def-abstr-PI}.
Theorem~\ref{thm-intro-PI} and its general version,
Theorem~\ref{thm-two-weight-PI-eps}, are proved there.
The self-improving argument is based on maximal functions as in 
Haj\l asz--Koskela~\cite[Theorem~5.3]{HaKobook}
and Heinonen--Koskela~\cite[Lemma~5.15]{HeKo98}, together with
Maz\cprime ya's truncation technique from \cite{Maz}. 

The Poincar\'e inequalities derived in Section~\ref{sect-derive-PI}
are further used in Section~\ref{sect-loc-doubl-meas} to prove 
the general compactness Theorem~\ref{thm-comp-for-q'<q-gen} and Corollary~\ref{cor-comp-dim}, 
assuming the weaker inequalities
\eqref{eq-def-abstr-PI}, together with more precise information about the
measures $\mu$ and~$\nu$. 
The spaces $P^{\al,p}_{\tau,E}$ considered in these statements, 
with Poincar\'e inequalities
required only for balls centred in $E$,  
could be of independent interest.
The spaces $M^{\al,p}$ and $P^{\al,p}_\tau$ 
are included as special cases.
Section~\ref{sect-Newt} deals with Newtonian spaces $\Np$ and the
sufficiency part of Theorem~\ref{thm-Mazya-cond-intro-new} is proved
there, while various necessary conditions for compact/bounded embeddings are obtained in 
Section~\ref{sect-optimality}.

Section~\ref{sect-examples} is devoted to concrete examples and to
embeddings on uniform domains with distance weights.
Example~\ref{ex-d-Hausdorff} deals with lower-dimensional measures and
Proposition~\ref{prop-bdry} is proved in this section.
Examples~\ref{ex-trace-emb} and \ref{ex-weighted-von-Koch} 
recover trace embeddings for Lipschitz domains in $\R^n$ and 
for the von Koch snowflake domain.

\begin{ack}
 J.~B. was supported by
the Swedish Research Council grants 621-2011-3139,
621-2014-3974 and 2018-04106. 
The work of A.~K. was supported by the National Science Center (Poland), grants
N201 397837 (years 2009-2012) and 2014/14/M/ST1/00600.
We are grateful to the anonymous referee for detailed comments that have 
substantially improved the paper and elliminated some early mistakes.
\end{ack}

\section{Compactness  and boundedness of embeddings} 
\label{sect-comp}

In this section we prove a general result about compact 
embeddings from certain function
spaces into $L^q$ spaces, possibly with respect to a different measure.
Unless said otherwise, we assume the following setting throughout the paper.

\begin{enumerate}
\renewcommand{\theenumi}{\textup{\bf (\Alph{enumi})}}%
\renewcommand{\labelenumi}{\theenumi}
\setcounter{enumi}{0}
  \item
\emph{General assumptions}: 
\end{enumerate}
\begin{enumerate}
\item[$\bullet$] $1\le p,q<\infty$,
\item[$\bullet$] $X$ is a   metric space equipped with a metric~$d$ and positive complete
Borel measures~$\mu$ and $\nu$, called the \emph{domain measure} and the
\emph{target measure}, 
\item[$\bullet$] $E\subset X$ is $\nu$-measurable and $\nu(E)>0$.
\end{enumerate}
Note that $X$ need not be complete.
For example, $X$ can be an open set in $\R^n$, as in 
Examples~\ref{ex-slit-disc}, \ref{ex-Rn-Rellich-Kondr} 
and~\ref{ex-bddness-cusp}.

We wish to cover the following possible situations: 
\begin{enumerate} 
\item[$\bullet$] $E$ is a domain (i.e.\ an open connected set)
in $X$ and $\nu$ is a restriction to $E$
of some  Borel measure defined on $X$ 
(such as $\mu$), 
\item[$\bullet$] $E$ is the boundary of a domain in $X$ and $\nu$ is
a lower dimensional Hausdorff measure defined on such a boundary.
\end{enumerate} 
If $\nu$ is a positive complete Borel measure defined only
on $E$, then we extend it to $X$ by setting $\nu(A):=\nu(A\cap E)$, 
whenever $A\cap E$ is $\nu$-measurable;
we will not distinguish between $\nu$ and its extension. 
We use the following notation: For a ball 
\[
B=B(x,r):=\{y\in X:d(x,y)<r\}, 
\]
let $\la B:=B(x,\la r)$.
In metric spaces, it can happen that balls with different
centres and/or radii denote the same set.
We  therefore adopt the convention that a ball is open and comes with a predetermined
centre and radius.
For $A\subset X$ with $0<\mu(A)<\infty$, let 
\begin{equation}    \label{eq-average}
u_{A,\mu}:=\vint_A u \, d\mu := \frac{1}{\mu(A)}\int_A u \,\dmu,
\end{equation}
and similarly for any other measure. 
Lebesgue spaces are denoted by $L^p(X,\mu)$,
where $\mu$ is often omitted from the notation if it is meant to be the Lebesgue
measure on a Euclidean space. 
The same simplification will be used for Sobolev type spaces. 
As in \cite{Haj-ContMat}, by $u\in L^p\loc(X,\mu)$ we mean that 
$u\in L^p(B,\mu)$ for all balls $B$ in $X$.

The letter $C$ denotes various positive
constants whose exact values depend only on unimportant
parameters and may vary even within the same line.

\subsection{General compactness and boundedness results}

Let ${\cal X}$ be a normed space and $\Y \subset {\cal X}$.
 Recall that a mapping from $\Y$  
into a Banach  space ${\cal Z}$ is  \emph{compact} if 
the image of every bounded sequence from 
$\Y$ has a convergent subsequence with
respect to the norm on ${\cal Z}$.
Our primary interest is in situations when $\Y$ is a subset of some
space ${\cal X}$ of functions defined on a metric space $X$.
For example, ${\cal X}$ can be  
a Sobolev space, which has a natural norm associated with it.
Since $\Y$ need not be a linear space, the restriction of such a norm need
not be a norm on $\Y$, only a metric.
We will, however, often use the term norm for it as well, 
in order to distinguish 
it from the metric $d$ defined on the underlying metric space $X$, 
which carries the functions under consideration.

One reason why we formulate our result in terms of the subset $\Y$,
rather than the normed space ${\cal X}$ itself, is that in applications,
Assumptions~\ref{ass-bdd} and \ref{ass-comp} below might only be satisfied 
for a certain class of functions, coming with suitable a priori estimates. 
Such a class need not be a linear space.
We use the notation  $\Y \embed {\cal Z}$ for the identity mapping, whenever
$\Y\subset{\cal Z}$, but also when suitable restrictions or
extensions of functions from $\Y$ belong to ${\cal Z}$, as in the trace
results in Proposition~\ref{prop-bdry}.

Given the exponents $1\le p,q<\infty$ and a set $\Y$ 
of $\nu$-measurable
functions defined on the metric space $X$, we consider the
following assumptions, where
\ref{ass-comp} stands  for compactness and \ref{ass-bdd} for boundedness:

\medskip

\begin{enumerate}
\renewcommand{\theenumi}{\textup{\bf (\Alph{enumi})}}%
\renewcommand{\labelenumi}{\theenumi}
\setcounter{enumi}{2}
  \item
\label{ass-comp} 
\emph{Existence of a countable sequence of finite covering families 
with a controlled overlap
and a subordinate Poincar\'e type inequality valid for all functions in~$\Y$}:
\end{enumerate}
\begin{itemize}
\item
For every $m\in\N$, there is a \emph{finite covering family}
$\{E_i,E'_i\}_{i=1}^{K_m}$, such that
\begin{enumerate}
\item[(i)] $E_i=E_i(m)\subset X$ are $\nu$-measurable, 
$\nu\bigl(E\setm\bigcup_{i=1}^{K_m}E_i\bigr)=0$ and 
\begin{equation}  \label{eq-nu(E-i)>0}
0<\nu(E_i)<\infty, \quad \text{for all }i=1,2,\ldots,K_m,
\end{equation}
\item[(ii)] $E'_i=E'_i(m)\subset X$ are $\mu$-measurable.  
\end{enumerate}
\item
The following  \emph{Poincar\'e type inequality} 
holds for each fixed $m$
and the corresponding sets $E_i$ and $E'_i$, $i=1,\ldots,K_m$:
There exist mappings $a_{E_i}:\Y\to\R$ and $G_m:\Y\to L^p(X,\mu)$ such that
$G_m(u)\ge 0$ $\mu$-a.e.\  and
\begin{equation}          \label{eq-def-loc-PI}
   \biggl(\int_{E_i} |u-a_{E_i}(u)|^q \,d \nu \biggr)^{1/q}
        \le C_m  \biggl( \int_{E'_i} G_m(u)^p \,\dmu \biggr)^{1/p},
\end{equation} 
whenever $u\in \Y$. 
The constant $C_m$ can depend on $m$, $p$, $q$ and on the family
$\{E_i,E'_i\}_{i=1}^{K_m}$, but not on $u$ and $i$.
\item
For such a covering family $\{E_i,E'_i\}_{i=1}^{K_m}$  we define the overlap
\begin{equation}  \label{eq-def-N}
N_m := \esssup_{x\in X} \sum_{i=1}^{K_m}  \chi_{E'_i}(x),
\end{equation}
where $\chi_A$ is the characteristic function of a set $A$
and the $\esssup$ is taken with respect to $\mu$.
\end{itemize}

The requirement that $\nu(E_i)>0$ in \eqref{eq-nu(E-i)>0} can clearly 
be fulfilled by discarding some of the $E_i$'s.
The sets $E_i$ and $E_i'$, as well as the mappings 
$u\mapsto a_{E_i}(u)$ and $u\mapsto G_m(u)$, may depend on $m$.
Also the overlap $N_m$ is allowed to depend on $m$, 
as long as it is compensated by the Poincar\'e constant $C_m$ 
in~\eqref{eq-def-loc-PI}, see Theorem~\ref{thm-precompact}.

\begin{remark} [\emph{Typical situations}]  
Concrete examples are when $E=X$  and $\nu=\mu$, 
or when $X=\clOm$ and $E=\bdry\Om$, where $\Om$
is a bounded open set,
equipped  with suitable measures. 
With $\Y$ equal to an appropriate Sobolev or fractional space,
our results then imply sharp 
compact Sobolev and trace type embeddings for such spaces.

Note that $E_i$ need not be a subset of $E_i'$.
For example, in trace theorems it may be convenient to have $E_i\subset\bdy\Om$
and $E_i'\subset\Om$ for some open set $\Om$ and $X=\clOm$.

Typical examples of $G_m(u)$ are $|\grad u|$ or its metric space
analogue $g_u$, but other choices are possible. 
Neither linearity nor locality is required for $G_m(u)$.
The possible dependence of $G_m(u)$ on $m$ allows
gradients adapted to various scales, such as the ones
introduced for fractional Haj\l asz, Besov and Triebel--Lizorkin spaces on metric spaces
 in Koskela--Yang--Zhou~\cite{KoYaZh11}.
We omit such generalizations here.
\end{remark}

\begin{example}[\emph{Choice of covering sets}]    \label{ex-slit-disc}
The simplest choice of a covering family is to let $E_i$ and $E'_i$ 
be balls with fixed radius $r_m$, such
  that $r_m\to0$ as $m\to\infty$.
The Poincar\'e type inequality~\eqref{eq-def-loc-PI} 
(as well as~\eqref{eq-def-loc-PI-bdd} below)
then often follows from the structure of the
set $\Y$ or from the usual \p-Poincar\'e inequality on $X$.  
However, this choice would exclude e.g.\ the following natural
situation: 

Let $X=E$ be the slit disc (i.e.\  
a disc with a radius removed) in the plane.
Covering $E$ by balls defined in $X$ by the Euclidean metric 
inevitably leads to disconnected balls divided by the slit.
Such balls do not support any Poincar\'e inequality for functions in
$W^{1,p}(X)$ with $G_m(u)= |\grad u|$,
while their connected components can be considered as two sets $E_i$
in our general assumptions.

For other types of gradients or other classes $\Y$ of 
functions it might still be possible that some kind of Poincar\'e
inequality holds on such disconnected balls.
Alternatively, connected balls with respect to the inner metric in $X$ can be used.
\end{example}

The number of covering families in \ref{ass-comp} is countable but 
each family is finite.
In general, the cardinality $K_m$ will tend to $\infty$, as $m\to\infty$. 
The following assumptions for boundedness require only one
covering family, which can be infinite.

\bigskip

\begin{enumerate}
\renewcommand{\theenumi}{\textup{\bf (\Alph{enumi})}}%
\renewcommand{\labelenumi}{\theenumi}
\setcounter{enumi}{1}
  \item
\label{ass-bdd} 
\emph{Existence of at most countable  covering family with a controlled overlap
and a subordinate Poincar\'e type inequality valid for all functions in} $\Y$:
\end{enumerate}
\begin{itemize}
\item
There exists a \emph{covering family}
$\{E_i,E'_i\}_{i\in\II}$ with $\II\subset\N$, such that
\begin{enumerate}
\item[(i)] $E_i\subset X$ are $\nu$-measurable, 
$\nu\bigl(E\setm\bigcup_{i\in\II}E_i\bigr)=0$,
\[
\nu(E_i)<\infty \quad \text{for every } i\in \II, \quad \text{and} 
\quad \inf_{i\in\II} \nu (E_i\cap E)>0,
\]
\item[(ii)] $E'_i\subset X$ are $\mu$-measurable. 
\end{enumerate}
\item
The following \emph{Poincar\'e type inequality} holds for all
$E_i$ and $E'_i$, $i\in\II$:
There exist mappings $a_{E_i}:\Y\to\R$ and $G:\Y\to L^p(X,\mu)$ such that
$G(u)\ge 0$ $\mu$-a.e.\  and
\begin{equation}         \label{eq-def-loc-PI-bdd}
   \biggl(\int_{E_i} |u-{a_{E_i}(u)} 
|^q \,d \nu \biggr)^{1/q}
        \le C_E \biggl( \int_{E'_i} G(u)^p \,\dmu \biggr)^{1/p},
\end{equation}
whenever $u\in \Y$. 
The constant $C_E$ can depend on $p$, $q$ and on the family
$\{E_i,E'_i\}_{i\in\II}$, but not on $u$ and $i$.
\item
The sets $E_i$ and $E'_i$, ${i\in\II}$, have a bounded overlap 
\begin{equation*}  
N := \esssup_{x\in X} \sum_{i\in\II} \chi_{E_i}(x) + \esssup_{x\in X} \sum_{i\in\II} \chi_{E'_i}(x) <\infty,
\end{equation*}
where the essential suprema are taken with respect to 
$\nu$ and $\mu$, respectively.
\end{itemize}

Our first result reads as follows. 
We postpone the proof to Section~\ref{sect-pf-comp}.

\begin{thm}[General sequential compactness]  \label{thm-precompact}  
Let $E$ and $\Y$ be such that
Assumptions~\ref{ass-comp} are satisfied  
with $C_m$, $N_m$, $K_m$, $G_m(\cdot)$ and $a_{E_i}(\cdot)$.
Let $\{u_n\}_{n=1}^\infty$ be a sequence in $\Y$, such that 
the sequence $\{a_{E_i}(u_n)\}_{n=1}^\infty$ is bounded
for every fixed $m$ and $i=1,\dots, K_m$.
Assume that one of the following conditions holds\/{\rm:}
\begin{enumerate}
\renewcommand{\theenumi}{\textup{(\roman{enumi})}}%
\renewcommand{\labelenumi}{\theenumi}
\item   \label{it-p-le-q} 
$1\le p\le q <\infty$  and 
\begin{equation}    \label{eq-CMN-to-0}
\lim_{m\to\infty} C_m N_m^{1/p} \sup_{n} \|G_m(u_n)\|_{L^p(X,\mu)} = 0,
\end{equation}  
\item    \label{it-p-ge-q}
$1\le q<p <\infty$ and   
\begin{equation}    \label{eq-CMN-to-00}
\lim_{m\to\infty} C_m N_m ^{1/p} K_m^{1/q-1/p} \sup_{n} \|G_m(u_n)\|_{L^p(X,\mu)} =0.
\end{equation}
\end{enumerate}
Then $\{u_n\}_{n=1}^\infty$ has a subsequence converging in $L^q(E,\nu)$.
\end{thm}

\begin{example}[\emph{Rellich--Kondrachov theorem}]  
\label{ex-Rn-Rellich-Kondr}
If $X=\R^n$ with the Euclidean metric 
and $\mu=\nu=dx$ is the Lebesgue measure, then 
the Sobolev--Poincar\'e inequality
\begin{equation}    \label{eq-(q,p)-PI}
\biggl( \vint_B |u-u_{B,dx}|^q\,dx \biggr)^{1/q}
\le C r \biggl( \vint_B |\grad u|^p\,dx \biggr)^{1/p}
\end{equation}
holds for all $q\le p^*:=np/(n-p)$, every ball $B=B(x_0,r)$
and every $u\in W^{1,p}(\R^n)$, where $p<n$.
This implies that the Poincar\'e type
inequality~\eqref{eq-def-loc-PI} is satisfied with 
$E_i=E'_i=B(x_{i,m},r_m)$ for some choice of $x_{i,m}\in E\subset\R^n$ and with 
\[
C_m = C r_m^{1+n/q-n/p} \to 0 \quad \text{as }r_m\to0 
\text{ when } q<p^*.
\]
Clearly, $E$ can be covered by such balls with a bounded overlap
only depending on~$n$.
For each fixed $E_i$ and for $a_{E_i}(u)=u_{E_i,dx}$ we have
the uniform bound
\begin{equation*}   
|a_{E_i}(u)|\le \mu({E_i})^{-1/p} \|u\|_{L^p(\Rn)}.
\end{equation*} 
Theorem~\ref{thm-precompact} thus implies the classical
Rellich--Kondrachov result about compactness of
the embedding $W^{1,p}(\R^n)\embed L^q(E)$ for
every bounded measurable set $E$ when $q<p^*$.
For bounded domains $\Om\subset\R^n$, compactness of the embedding 
\[
W^{1,p}(\Om) \longemb L^q(\Om)
\] 
can be obtained in a similar way if the intersections $\Om\cap B(x,r_m)$,
$x\in\Om$, support~\eqref{eq-def-loc-PI}.
This is possible e.g.\ for uniform domains, and thus for bounded Lipschitz
domains, see Remark~\ref{rem-when-mu|Om}, 
Aikawa~\cite[p.~120]{Aik}
and Maz$'$ya~\cite[Sections~1.1.8--1.1.11]{MazSobBook}.
\end{example}

\begin{remark}[\emph{Discussion of sharpness}]   \label{rem-Haj-Liu-opt}
For certain normed spaces,
compactness of an embedding into $L^q(X,\mu)$ is equivalent to
the boundedness of an embedding into a ``better'' Orlicz space 
$L^\Phi(X,\mu)$, provided that $\mu(X)<\infty$, see Haj\l asz--Liu~\cite{hajzu}.
On the other hand, in Remark~\ref{rem-limiting-case} and
Example~\ref{ex-optimal-emb} we present
a compact embedding with a target space that is optimal among $L^q$
spaces.

Theorem~\ref{thm-Mazya-cond-intro-new} shows that a sufficient
condition on measures, deduced from Theorem~\ref{thm-precompact},
is also essentially necessary for compactness. 
A measure which exactly satisfies this condition for the optimal
exponent is presented in Example~\ref{ex-optimal-emb}. 

In Theorem~\ref{thm-precompact} we prove compactness in $L^q$ for the same exponent $q$ as
in the left-hand side of~\eqref{eq-def-loc-PI}, provided that  $C_m$
and $N_m$ are  well-controlled as $m\to\infty$.
This is in contrast to Chua--Rodney--Wheeden~\cite{ChuaRodWhe},  which uses uniformly
bounded overlap (as in Euclidean spaces)
and obtains compactness in $L^{q'}$ only for $q'<q$, assuming initial boundedness in $L^q$.
For $q\ge p$ we need stronger Poincar\'e inequalities than~\cite{ChuaRodWhe},
 but we can also
reach the limiting exponent $q$ if $C_m$ is good enough.
This is exhibited e.g.\ in Theorem~\ref{thm-w-al-v-be}, Example~\ref{ex-optimal-emb}
and Proposition~\ref{prop-q=p-M-al}.
\end{remark}

Our next result deals with the  boundedness  of embeddings.
Note that finiteness is not required for $\mu$ and $\nu$, not even
locally.
For instance, in Example~\ref{ex-bddness-cusp}, finiteness of the measure fails
for balls containing the origin.

\begin{thm}[General boundedness]  \label{thm-bounded}  
Let $E$ and $\Y$ be  such that
the Assumptions~\ref{ass-bdd} are satisfied.  
Assume that one of the following conditions holds\/{\rm:}
\begin{enumerate}
\renewcommand{\theenumi}{\textup{(\roman{enumi})}}%
\renewcommand{\labelenumi}{\theenumi}
\item 
The index set $\II$ is finite, $1\le q<\infty$ and $t\ge1$, 
\item 
$\II$ is countably infinite, $1\le p\le q<\infty$ and $1\le t<q$.
\end{enumerate}
Then there is $C>0$ such that for all $u\in \Y\cap L^t(E,\nu)$,
\[
\| u\|_{L^q(E,\nu)} \le C \bigl( \| G(u)\|_{L^p(X,\mu)} 
    +\| u\|_{L^t(E,\nu)} \bigr).
\]
\end{thm}

Also the proof of Theorem~\ref{thm-bounded} is postponed to 
Section~\ref{sect-pf-comp}.

\begin{example}[\emph{Bounded embeddings on bad domains}]             
\label{ex-bddness-cusp}
Consider the cusp 
\[
\Om:=\{x=(x',x_n)\in\R^n: |x'|<x_n^\ga <1 \}, \quad \ga>1,
\]
and let $X=E=\Om$ with the Euclidean metric.
For each $k=1,2,\ldots$\,, let $j_k$ be the smallest integer such that
$j_k\ge 2^{k(\ga-1)}$ and cover $\Om$ up to a set of zero measure 
by countably many ``chunks''
\[
\Omkj :=\{ x\in \Om: 2^{-k}+ (j-1)2^{-k\ga} <x_n< 2^{-k}+j2^{-k\ga}\}, 
\quad j=1,2,\ldots,j_k,
\]
of length $2^{-k\ga}$, $k=1,2,\ldots$\,.
Each $\Omkj$ is biLipschitz equivalent to the unit cylinder 
\[
T_k(\Omkj) = B(0,1) \times (j-1,j) \subset \R^n, \quad j=1,2,\ldots,j_k,
\]
by means of the mapping 
$T_k:(x',x_n)\mapsto \bigl( x'/x_n^{\ga} ,2^{k\ga}(x_n - 2^{-k})\bigr)$.
The triangle inequality implies that
\[
|T_k(x)-T_k(y)| \le \frac{1}{x_n^{\ga}} \biggl( |x'-y'| 
   + \frac{|y'|}{y_n^{\ga}} |x_n^{\ga}-y_n^{\ga}| \biggr) + 2^{k\ga}|x_n-y_n| 
\le C 2^{k\ga} |x-y|.
\]
Similar estimates hold also for $T_k^{-1}$ and we conclude that 
$|T_k(x)-T_k(y)|$ is comparable to $2^{k\ga}|x-y|$.
Consider the measures 
\[
d\mu (x) =x_n^{\alpha}\, dx \quad \text{and} \quad
d\nu = x_n^{\beta}\, dx,
\]
which on each $\Omkj$ are comparable to $2^{-k\al}\,dx$ and
$2^{-k\beta}\,dx$, respectively.
Starting from the classical Sobolev--Poincar\'e inequality on the unit 
cylinder $T_k(\Omkj)$, 
we arrive after a suitable change of variables at
the Sobolev--Poincar\'e type inequality
\begin{equation*}
\biggl( \int_{\Omkj} |u-a_{k,j}(u)|^q \,d\nu \biggr)^{1/q}
\le C 2^{-k\theta} \biggl( \int_{\Omkj} |\grad u|^p \,d\mu \biggr)^{1/p},
\end{equation*}
where $C$ does not depend on $k$ and $j$,
\begin{equation}    \label{eq-def-akj}
a_{k,j}(u)=\vint_{T_k(\Omkj)}u\circ T_k^{-1} (y)\,dy,
\end{equation}
\[
\theta = \ga +\frac{\beta+n\ga}{q} - \frac{\alpha +n\ga}{p}, 
\quad q\le p^*=\frac{np}{n-p}, \quad p<n.
\]
For simplicity,  choose $\al=\beta =-n\ga$ which
gives $\theta \ge 0$ and $1/C \le \nu (\Omkj) \le C$ for all $k,j$.
Theorem~\ref{thm-bounded}, with $E_i$ and $E'_i$ replaced by $\Omkj$,
now guarantees boundedness of the embedding  
\begin{equation}  \label{eq-emb-on-cusp}
W^{1,p}(\Omega,\mu) \longemb  L^q(\Omega,\nu), 
\quad p\le q\le p^*.
\end{equation}
Since $\nu(\Om)=\infty$, we cannot conclude embedding for $q<p$.
In fact, for 
\[
 q< \frac{p(\ga-1)}{p+\ga-1},
\] 
the function $u(x)=x_n^{(\ga-1)/q}$
shows that the embedding $W^{1,p}(\Om,\mu) \embed  L^q(\Om,\nu)$
fails. 
As $\ga\to\infty$, this includes all $q<p$ and hence the above range $q\ge p$
 is at least asymptotically sharp.
Since $\mu$ and $\nu$ are comparable to a multiple of the Lebesgue measure on 
each $\Omkj$, the upper end point $p^*$ in \eqref{eq-emb-on-cusp} is sharp
in the same way as for unweighted Sobolev spaces.

We leave it to the interested reader to make the necessary 
modifications for other weights and domains.
A similar argument with $\ga=1$ can be applied to the
punctured ball $B(0,1)\setm\{0\}\subset\R^n$, equipped with the 
measure $d\mu(x)=|x|^{-n}\,dx$, 
where the role of the sets $E_i$
is played by Whitney cubes near the origin.
Whitney type decompositions of other domains
can also be considered.
\end{example}

\subsection{Discussion of Assumptions~\ref{ass-bdd}
and~\ref{ass-comp}}
\label{sect-discuss}

Our aim is to obtain compactness 
and boundedness under least possible
assumptions on the covering family and the involved measures.
Several remarks are therefore in order to clarify these assumptions.

\begin{remark}[\emph{Choice of  $a_{E_i}(u)$}]       \label{rem-choices}
A standard choice in~\eqref{eq-def-loc-PI} 
and \eqref{eq-def-loc-PI-bdd} is $a_{E_i}(u)= u_{E_i,\nu}$ 
but other integral averages, medians, traces or even nonlinear
functionals can also be used, see e.g.\ \eqref{eq-def-akj} in
Example~\ref{ex-bddness-cusp} and the proof of
Theorem~\ref{thm-w-al-v-be}.

An application of the triangle inequality shows
that if \eqref{eq-def-loc-PI} and
  \eqref{eq-def-loc-PI-bdd} hold with some $a_{E_i}(u)$ then they also
hold with $a_{E_i}(u)=u_{E_i,\nu}$, whenever defined,
at the cost of enlarging the constant on the right-hand side.
\end{remark}

\begin{remark}[\emph{Covering by balls}] \label{rem-Hausdorff-max-princ}
Since $E\subset\bigcup_{x\in E}B(x,r)$, the Hausdorff maximality principle
(see e.g.\ Ziemer~\cite[p.\ 7]{Ziem}) 
provides us for every $r>0$ with a maximal
pairwise disjoint family of balls $\tfrac12 B_\al$ 
of radius $\tfrac12r$ and centres in $E$.
The maximality of the family implies that the balls $B_\al$ cover $E$.
In general, this construction
does not guarantee any bounds on the overlap of the balls $B_\al$,
as shown by the following example.
\end{remark}

\begin{example}[\emph{Unbounded overlap}]   
Let $X=\bigcup_{n=1}^\infty X_n$,  where $X_n=[0,2^{-n}]^n\subset\R^n$,
$n=1,2,\ldots$\,, and the origins in all $X_n$ are identified as one point $0$
in $X$.
The metric on $X_n$ is defined as the $n$-dimensional Euclidean metric $d_n$, 
while for $x\in X_n$ and $y\in X_m$ with $n\ne m$ we let $d(x,y)=d_n(x,0)+d_m(y,0)$.
Note that $X$ is compact.
We shall now see how $E:=X$ can be covered by open balls with radius 
$r=2^{-m}$, which will give us the covering 
family $\{ E_i,E_i^{'}\}_{i=1}^{K_m}$. 

The ball $B(0,r)$ contains
all $X_n$ with $2^{-n}\sqrt{n}\le r$, while the
remaining $X_n$ have to be covered by additional balls of radius $r$.
For such $n$ and sufficiently large $m$, the required overlap of Euclidean balls is determined by
the Lebesgue covering dimension of $\R^n$, 
and equals $n+1$, see Munkres~\cite[p.~305]{Munkres}.
Since the minimal overlap $N_m$ on $X$ majorizes the minimal overlap for all
such $X_n$ with $2^{-n}\sqrt{n}> r$, we conclude that $N_m\to\infty$ as $m\to\infty$.
\end{example}

We shall now see that a doubling property guarantees a bounded overlap.

\begin{deff}[\emph{Doubling spaces and measures}]
\label{def-doubling}
A metric space $X$ is 
\emph{doubling}\/ if every ball of radius $r$ 
can be covered by at most
$M$ balls of radius $r/2$, where $M$ is independent of $r$.
A measure $\mu$ is \emph{doubling} if there exists a constant
$C_\mu \geq 1$ such that for all balls $B$ in~$X$,
\begin{equation*} 
0<\mu(2B) \le C_\mu \mu(B)<\infty.
\end{equation*} 
\end{deff}

Every space carrying a doubling measure is doubling.
The following lemma, which may be of independent interest, 
implies that if $X$ 
or $\mu$ is doubling and the covering family is obtained from balls 
with the same radius then such a family can always be chosen with 
a bounded overlap independent of the radius.

The doubling property also implies that such a family is finite
whenever it is confined to a bounded set.
In particular, every bounded set is totally bounded.
This need not be true if the doubling property only holds   
for balls up to a certain radius. 
Thus, total boundedness of $E$ will be explicitly assumed
in those cases.

\begin{lem}[Doubling and controlled overlap] 
\label{lem-Hausdorff-max-princ}
Assume that $\la\ge\tfrac12$ and that for a function 
$M_\la:(0,\infty)\to(0,\infty)$, 
one of the following conditions holds for all $r>0$\/{\rm:}
\begin{enumerate}
\item  \label{it-M(r)-X}
Every ball of radius $\la r$  can be covered by  
$M_\la(r)$ balls of radius $\tfrac12r$.
\item  \label{it-M(r)-mu}
For all balls $B$ of radius $\tfrac12r$,
\begin{equation*}  
0<\mu((4\la+1)B) \le M_\la(r) \mu(B) <\infty.
\end{equation*} 
\end{enumerate}
\smallskip
Let $\{B_\al=B(x_\al,r)\}_{\al\in A}$
be a collection of balls  such that the balls $\{\tfrac12 B_\al\}_{\al\in A}$
are pairwise disjoint.
Then the balls $\{\la B_\al\}_{\al\in A}$  have overlap at most $M_\la(r)$.

In particular, for doubling spaces and measures, the overlap 
is independent of~$r$.
\end{lem}

\begin{proof}
Let $x\in X$ be arbitrary and let $A_x=\{\al\in A:x\in \la B_\al\}$.
Then
\begin{equation}   \label{eq-Bal-subset-la+1/2}
x_\al\in B(x,\la r) \quad \text{and} \quad 
\tfrac12 B_\al\subset B(x,(\la+\tfrac12)r) \quad \text{for all } \al\in A_x.
\end{equation}
Suppose that \ref{it-M(r)-X} holds. 
Then the ball $B(x,\la r)$ can be covered by  
$M_\la(r)$  balls $B'_i$ of radius $\tfrac12r$.
Hence, every $x_\al$ with $\al\in A_x$ belongs to some $B'_i$,
whose centre in turn belongs to $\tfrac12 B_\al$.
Since the balls $\{\tfrac12 B_\al\}_{\al\in A}$ are pairwise disjoint, 
it follows that $x_\al$ and $x_{\al'}$ cannot belong to the same $B'_i$ 
and consequently, there are at most $M_\la(r)$ indices
$\al$ such that $x\in \la B_\al$.
Similarly, if \ref{it-M(r)-mu} holds then for every $\alpha \in A_x$, 
\[
\mu(B(x,(\la+\tfrac12)r)) \le \mu(B(x_\al,(2\la+\tfrac12)r))
\le M_\la(r) \mu(\tfrac12 B_\al).
\]
Hence, by \eqref{eq-Bal-subset-la+1/2} and
the pairwise disjointness of the balls $\tfrac12 B_\al$, $\al\in A$, 
\[
\mu(B(x,(\la+\tfrac12)r)) \ge \sum_{\al\in A_x} \mu(\tfrac12 B_\al) \ge
\frac{\mu(B(x,(\la+\tfrac12)r))}{M_\la(r)} \sum_{\al\in A_x} 1,
\]
from which it follows that the index set $A_x$ can have at most $M_\la(r)$ elements.
\end{proof}

\subsection{Proofs of Theorems~\ref{thm-precompact} 
and~\ref{thm-bounded} }
\label{sect-pf-comp}

\begin{proof}[Proof of Theorem~\ref{thm-precompact}] 
We start by considering the case~\ref{it-p-le-q}.
For every fixed $m=1,2,\ldots$,
consider the covering family $\{E_i,E'_i\}_{i=1}^{K_m}$.
Let $u\in \Y$ be arbitrary, with $G_m(u)$ as in~\eqref{eq-def-loc-PI}.  
Then,  by the local Poincar\'e type inequality~\eqref{eq-def-loc-PI}, 
\begin{equation}  \label{eq-after-loc-PI}
\sum_{i=1}^{K_m} \int_{E_i} |u-a_{E_i}(u)|^q \,d\nu
\le C_m^q \sum_{i=1}^{K_m} \biggl( \int_{E'_i} G_m(u)^p \,d\mu \biggr)^{q/p}.
\end{equation}
Since $q\ge p$, the elementary inequality
$\sum_i x_i^{q/p} \le \bigl(\sum_i x_i\bigr)^{q/p}$ yields
\begin{equation*}  
\sum_{i=1}^{K_m} \biggl( \int_{E'_i} G_m(u)^p \,d\mu \biggr)^{q/p}
\le 
\biggl( \displaystyle \sum_{i=1}^{K_m} \int_{E'_i} G_m(u)^p \,d\mu \biggr)^{q/p}.
\end{equation*}
Inserting this into \eqref{eq-after-loc-PI}, together with~\eqref{eq-def-N},
implies
\begin{equation}  \label{eq-est-with-C-K-N}
\sum_{i=1}^{K_m} \int_{E_i} |u-a_{E_i}(u)|^q \,d\nu
\le C_m^q \biggl( N_m \int_{X} G_m(u)^p \,d\mu \biggr)^{q/p}.
\end{equation}
Now, consider a sequence $\{u_n\}_{n=1}^\infty\subset\Y$,
such that $\{G_m(u_n)\}_{n=1}^\infty$ satisfies~\eqref{eq-CMN-to-0} and
$\{a_{E_i}(u_n)\}_{n=1}^\infty$ is bounded for every fixed $E_i$
in the covering family.
We shall now show that $\{u_n\}_{n=1}^\infty$ has a Cauchy subsequence in
$L^q(E,\nu)$.
We have for all $n,k\ge1$,
\begin{align}   \label{eq-split-vn-vm}
&\int_E |u_n-u_k|^q\,d\nu
\le 3^{q-1} \sum_{i=1}^{K_m} \biggl( \int_{E_i} |u_n-a_{E_i}(u_n)|^q \,d\nu
\\
&\quad \quad \quad \quad \quad \quad \quad \quad \quad
+ \int_{E_i} |u_k-a_{E_i}(u_k)|^q \,d\nu
+ \int_{E_i} {|a_{E_i}(u_n)-a_{E_i}(u_k)|^q} \,d\nu \biggr). \nonumber
\end{align}
The sums of the first two integrals on the right-hand side can be estimated
using \eqref{eq-est-with-C-K-N} with $u$ 
replaced by $u_n$ and $u_k$, respectively.
Hence, because of~\eqref{eq-CMN-to-0}, 
we can for every $\eps>0$
choose a sufficiently large $m$ in \eqref{eq-split-vn-vm}
such that for all $n,k\ge1$,
\begin{equation}   \label{eq-fix-r}
\int_E |u_n-u_k|^q\,d\nu \le \eps
+ 3^{q-1} \sum_{i=1}^{K_m} |a_{E_i}(u_n)-a_{E_i}(u_k)|^q \nu(E_i),
\end{equation}
where $\{E_i\}_{i=1}^{K_m}$ is the covering family corresponding to $m$.
With this family fixed,
the sequence $\{a_{E_i}(u_n)\}_{n=1}^\infty$ is bounded in $\R$
for every $i=1,\ldots,K_m$.
Hence, applying the Bolzano--Weierstrass theorem, 
we can for $\eps_1=\tfrac12$ and a suitable family $E_i=E_i(m_1)$,
corresponding to $m_1$,
find a subsequence $\{u^{(1)}_n\}_{n=1}^\infty$ of $\{u_n\}_{n=1}^\infty$
such that the sequence $\{a_{E_i}(u^{(1)}_n)\}_{n=1}^\infty$
is convergent for every $i=1,\ldots,K_{m_1}$ and~\eqref{eq-fix-r} becomes
\[
\int_E |u^{(1)}_n-u^{(1)}_k|^q\,d\nu \le \frac12
+ 3^{q-1} \sum_{i=1}^{K_{m_1}} |a_{E_i}(u^{(1)}_n)-a_{E_i}(u^{(1)}_k)|^q \nu(E_i)
< 1     
\]
for all $n,k\ge1$.
Similarly, we can find another family 
corresponding to $\eps_2=\tfrac14$ and $m_2$, and a subsequence 
$\{u^{(2)}_n\}_{n=1}^\infty$ of $\{u^{(1)}_n\}_{n=1}^\infty$ such that
\[
\int_E |u^{(2)}_n-u^{(2)}_k|^q\,d\nu < \frac12 \quad \text{when } n,k\ge1.
\]
Continuing in this way and choosing the diagonal sequence
$\{u^{(n)}_n\}_{n=1}^\infty$ we construct
a Cauchy sequence in $L^q(E,\nu)$.
Since $L^q(E,\nu)$ is complete,  we are done with the case~\ref{it-p-le-q}.
In the case~\ref{it-p-ge-q}, the only change needed in the above proof 
is to use H\"older's inequality instead of 
the elementary inequality just after~\eqref{eq-after-loc-PI}.
\end{proof}

\begin{proof}[Proof of Theorem~\ref{thm-bounded}] 
By integrating over smaller sets in the left-hand side 
of~\eqref{eq-def-loc-PI-bdd}, if needed,
we can assume that $E=\bigcup_{i\in\II} E_i$. 
According to Remark~\ref{rem-choices} we can assume that 
$a_{E_i}(u)= u_{E_i,\nu}$ in~\eqref{eq-def-loc-PI-bdd}.
Then
\begin{align}   \label{eq-u-a}
\int_E |u|^q \,d\nu &\le 2^{q-1} \sum_{i\in\II} 
       \biggl( \int_{E_i} |u-a_{E_i,\nu}(u)|^q \,d\nu  
            + \int_{E_i} |u_{E_i,\nu}|^q \,d\nu \biggr) \nonumber\\
&\le 2^{q-1} \sum_{i\in\II} 
      \biggl( C_E^q \biggl( \int_{E'_i} G(u)^p \,d\mu \biggr)^{q/p}
            + \nu(E_i) \biggl( \vint_{E_i} |u|^t \,d\nu \biggr)^{q/t} \biggr),
\end{align}
where in the last term we used H\"older's  inequality.
When $\II$ is finite this immediately yields that
\[
\int_E |u|^q \,d\nu 
    \le C \bigl( \| G(u)\|^q_{L^p(X,\mu)} + \|u\|^q_{L^t(E,\nu)} \bigr).
\]
For infinite $\II$ and $q\ge p$, the elementary inequality
$\sum_i x_i^{q/p} \le \bigl(\sum_i x_i\bigr)^{q/p}$ yields
\[
\sum_{i\in\II} \biggl( \int_{E'_i} G(u)^p \,d\mu \biggr)^{q/p}
  \le \biggl( \sum_{i\in\II} \int_{E'_i} G(u)^p \,d\mu \biggr)^{q/p}
  \le N^{q/p} \| G(u)\|^q_{L^p(X,\mu)},
\]
because of the bounded overlap of the sets $E'_i$.
Since $q/t>1$, the lower bound on $\nu(E_i)$, the above 
elementary inequality
and the bounded overlap of $E_i$ imply that
\begin{equation*}
\sum_{i\in\II} \nu(E_i) \biggl( \vint_{E_i} |u|^t \,d\nu \biggr)^{q/t} 
   \le \Bigl( \inf_{i\in\II} \nu(E_i) \Bigr)^{1-q/t} 
         N^{q/t} \|u\|^q_{L^t(E,\nu)}.
\end{equation*}
Inserting the last two estimates into~\eqref{eq-u-a} concludes the proof
also for infinite~$\II$.
\end{proof}

\section{Embeddings with nondoubling measures}
\label{sect-nondoubl}

We now provide some concrete results about compact 
embeddings based on conditions \eqref{eq-CMN-to-0} and \eqref{eq-CMN-to-00}.
Our first result deals with the weighted spaces
\begin{align*}
\Wp(\Om,\mu) &= \{u\in L^p(\Om,\mu): |\grad u|\in L^p(\Om,\mu)\}, \\
\Dp(\Om,\mu) &= \{u\in L^1\loc(\Om,\mu): |\grad u|\in L^p(\Om,\mu)\},
\end{align*}
where $\grad u$ is the distributional gradient  of $u$ and
$d\mu=w\,dx$ with a $B_p$ weight $w$, i.e.\ $w^{1/(1-p)}\in L^1\loc(\R^n)$.
Such weights were introduced in Kufner--Opic~\cite{kuf-opic} and
are suitable for Sobolev spaces based on the distributional gradient $\grad u$.
See Ambrosio--Pinamonti--Speight~\cite{AmPiSp} and
Zhikov~\cite{Zhikov} for a discussion of various types of weighted Sobolev
spaces in $\R^n$ and on metric spaces.

\begin{thm}[Precise embeddings based on integrability conditions]   
\label{thm-w-al-v-be}
Let $w,v$ be weights such that $w, w^{-\al},v^\be \in L^1\loc(\R^n)$
with some  $\al>0$ and $\be>1$.
Assume that 
\[
\max \biggl\{ n\Bigl(\frac{1}{\al} + \frac{1}{\be}\Bigr),
1+\frac{1}{\al} \biggr\} \le p < n \Bigl(1+ \frac{1}{\al} \Bigr),
\]
and let $d\mu=w\,dx$ and $d\nu=v\,dx$.
Then the embeddings    
\[
\Wp(\R^n,\mu)\longemb L^q(E,\nu)
\quad \text{and} \quad
\Dp(\R^n,\mu)\cap L^1(E,\nu)\longemb L^q(E,\nu)
\]
are compact for every bounded 
Lebesgue measurable set $E\subset\R^n$ and all exponents
\begin{equation} \label{eq-range-for-q}
1\le q   \le  \frac{np (1-1/\be)}{(n-p)+n/\al}.
\end{equation}
\end{thm}

A similar result can be proved for weighted Sobolev spaces in metric
spaces with a well-behaved underlying measure.
We shall not dwell upon such generalizations.
The assumptions on $\al$ and $p$ imply that $w$ is a $B_p$ weight.

Example~\ref{ex-optimal-emb} below demonstrates the  sharpness of 
the limiting exponent in~\eqref{eq-range-for-q}.
Note that the exponent $q$   
has the correct asymptotics  $p^*= np/(n-p)$, as $\al,\be\to\infty$.
Moreover, the proof shows that if $w\ge C>0$,  then the statement
of Theorem~\ref{thm-w-al-v-be} holds for $1\le q \le p^*(1-1/\be)$
and similarly, it holds for $1\le q\le np/((n-p)+n/\al)$ when $v\le C'$.
This means that in~\eqref{eq-range-for-q},
we may also allow the limiting cases $\al=\infty$ or $\be=\infty$, 
but not both at the same time.
This case is excluded because then 
both integrals in~\eqref{eq-C-tilde} would be replaced by $L^\infty$-norms of
$1/w$ and $v$, and thus, $\Ct(B)$ would not be small for small balls $B$.

\begin{proof}
The result will be deduced from Theorem~\ref{thm-precompact} with $X=\R^n$.
We start by a verification of Assumptions~\ref{ass-comp}.
It clearly suffices to consider   
\begin{equation*}  
q = \frac{np (1-1/\be)}{(n-p)+n/\al}.
\end{equation*}
The covering family will consist of suitably chosen balls $E_i=E_i'=B$,
whose precise construction we postpone until the end of the proof.
To obtain the Poincar\'e type inequality~\eqref{eq-def-loc-PI} for
such balls,
let $u\in\Dp(\R^n,\mu)$ be arbitrary. 
The H\"older inequality implies that for $t=\al p/(\al+1)$ and every
ball $B\subset\R^n$,
\begin{equation}   \label{eq-est-w-al}
\biggl( \int_B |\grad u|^t\,dx \biggr)^{1/t}
\le \biggl( \int_B |\grad u|^p w\,dx \biggr)^{1/p}
           \biggl( \int_B w^{-\al}\,dx \biggr)^{1/\al p}
\end{equation}
and hence $u\in D^t(B,dx)$. 
Note that $1\le t<n$ and $q=t^*(1-1/\be)<t^*$, where 
\[
t^*=\frac{nt}{n-t} = \frac{np }{(n-p)+n/\al}    
\] 
is the Sobolev exponent associated with $t$.
Another use of the H\"older inequality, together with the usual
$(t^*,t)$-Sobolev--Poincar\'e inequality for the Lebesgue measure
in $\R^n$ and~\eqref{eq-est-w-al}, then yields 
for all balls $B\subset\R^n$ that
\begin{align}   \label{eq-VL-nu-HL-mu}
&\biggl( \int_B |u-u_{B,dx}|^q \,d\nu \biggr)^{1/q} 
\le \biggl( \int_B |u-u_{B,dx}|^{t^*} \,dx \biggr)^{1/t^*}
          \biggl( \int_B v^{\be}\,dx \biggr)^{1/\be q}  \\ 
& \quad \quad \quad
\le C(n,t) |B|^{1/t^*} \diam(B)\biggl( \frac{1}{|B|}\int_B |\grad u|^{t}\,dx 
          \biggr)^{1/t} \biggl( \int_B v^{\be}\,dx \biggr)^{1/\be q} \nonumber \\
&\quad \quad \quad
 \le C'(n,t) \biggl( \int_B |\grad u|^p w\,dx \biggr)^{1/p}
     \biggl( \int_B w^{-\al}\,dx \biggr)^{1/\al p}
     \biggl( \int_B v^{\be}\,dx \biggr)^{1/\be q}, \nonumber  
\end{align}
where $C(n,t)$ comes from the $(t^*,t)$-Sobolev--Poincar\'e inequality in $\R^n$.
The triangle inequality
allows us to replace $u_{B,dx}$  in~\eqref{eq-VL-nu-HL-mu}
by $u_{B,\nu}$ at the cost of an additional factor 2 on the right-hand side.
We have thus shown that for every ball $B\subset\R^n$,
\begin{equation}   \label{eq-PI-with-aB}
\biggl( \int_B |u-a_B(u)|^q \,d\nu \biggr)^{1/q} 
  \le \Ct(B) \biggl( \int_B |\grad u|^p w\,dx \biggr)^{1/p},
\end{equation}
where both possibilities $a_B(u)= u_{B,dx}$ and $a_B(u)= u_{B,\nu}$ are allowed and
\begin{equation}  \label{eq-C-tilde}
 \Ct(B) := 2C'(n,t)
\biggl( \int_B w^{-\al}\,dx \biggr)^{1/\al p}
     \biggl( \int_B v^{\be}\,dx \biggr)^{1/q-1/t^*}.
\end{equation}
Let $E$ be as in the statement of the theorem.
For each $m=1,2,\ldots$, we find a suitable covering family satisfying~\eqref{eq-CMN-to-0}. 
Since $w^{-\al},v^\be \in L^1\loc(\R^n)$,
we can for every $x\in \R^n$ find
a ball $B_x\ni x$ such that $\Ct(2B_x)<1/m$.
By compactness, the closure $\itoverline{E}$ can be covered by finitely many 
such balls $B_{x_j}$ with radii~$\rho_j$.
Choose $r_m\le \min\{1/m,\min_j\rho_j\}$.
Use the Hausdorff maximality principle as in
Remark~\ref{rem-Hausdorff-max-princ} and
Lemma~\ref{lem-Hausdorff-max-princ} to
cover $E$ by balls $B_i$ with radius $r_m$, centres in $E$
and a bounded overlap independent of $m$.
Since $r_m\le\rho_j$, we see that each $B_{i}$ is contained in some $2B_{x_j}$
and hence $\Ct(B_{i})\le \Ct(2B_{x_j})<1/m$.
These balls form a covering family for $E$ and~\eqref{eq-PI-with-aB}
shows that the Assumptions~\ref{ass-comp} are satisfied with 
$E_i=B_{i}\cap E$, $E_i'=B_i$,
\[
C_m = \max_{i} \Ct(B_{i})< 1/m, \quad
a_{E_i}(u)= u_{B_i,dx} \quad \text{and}  \quad a_{E_i}(u)= u_{B_i,\nu}.
\]
The condition $q\ge p$ is guaranteed by $p\ge n(1/\al+1/\be)$.
Finally, as in \eqref{eq-est-w-al},
\[
|u_{B_i,dx}| \le \biggl( \vint_{B_i} |u|^t \,dx \biggr)^{1/t}
   \le \biggl( \vint_{B_i} |u|^p w\,dx \biggr)^{1/p}
                    \biggl( \vint_{B_i} w^{-\al}\,dx \biggr)^{1/\al p},
\]
and thus the sequence of integral averages
$a_{E_i}(u_n):=(u_n)_{B_i,dx}$ is bounded whenever
$\{u_n\}_{n=1}^\infty$ is bounded in $\Wp (B,\mu)$.
If instead $\{u_n\}_{n=1}^\infty$ is bounded in $\Dp(B,\mu)\cap L^1(B,\nu)$, 
we use $a_{E_i}(u)= u_{B_i,\nu}$.
Since the overlap is independent of $m$, 
we see that~\eqref{eq-CMN-to-0} holds and
Theorem~\ref{thm-precompact} concludes the proof.
\end{proof}

\begin{remark}  \label{rem-limiting-case}
Let the notation be as in the proof of Theorem~\ref{thm-w-al-v-be}. 
The H\"older inequality shows, as in \eqref{eq-est-w-al}, 
that the following embeddings are bounded:
\[
\Wp (\R^n,\mu) \longemb  W^{1,t}(B,dx)
\quad \text{and} \quad
L^\tau(E)  \longemb L^{\tau(1-1/\be)}(E,\nu)
\]
for every ball $B\subset\Rn$ and every $\tau\ge1$.
Hence, by directly using the classical 
compact embedding $W^{1,t}(B,dx)\embed L^{\tau}(B)$,
we obtain compactness of the embedding
\[
\Wp (\R^n,\mu)\longemb L^{\tau(1-1/\be)}(E,d\nu)
\quad \text{whenever }\tau
< t^* {= \frac{n\al p}{(n-p)\al+n}},
\]
i.e.\ for $q<t^*(1-1/\be)$.
On the other hand, Theorem~\ref{thm-w-al-v-be}
makes it possible to reach also the limiting exponent $q=t^*(1-1/\be)$.
The following example shows that it
is optimal among all $q$'s for which one has a compact embedding into
$L^q(E,\nu)$.

Other results concerning compact embeddings for limiting exponents were 
under certain assumptions recently obtained 
in Gaczkowski--G\'orka--Pons~\cite{GaGoPons}. 
\end{remark}

\begin{example}   
[Optimal compactness]
\label{ex-optimal-emb}
Let $n=p=q=\al=\be=2$ in Theorem~\ref{thm-w-al-v-be} and set 
for $x\in\R^2$,
\begin{equation*} 
w(x) = \begin{cases} |x|\log(1/|x|), & \text{if }0<|x|<\tfrac12, \\
         \tfrac12 \log2, & \text{otherwise,}   \end{cases}
\qquad \text{and} \qquad v(x) = \frac{1}{w(x)}.
\end{equation*}
Elementary calculations show that the assumptions in 
Theorem~\ref{thm-w-al-v-be} are satisfied and that
for sufficiently small $r>0$,
\[
\mu(B(0,r)) \simeq r^3\log\frac1r 
\quad \text{and} \quad
\nu(B(0,r)) \simeq \frac{r}{\log(1/r)}.
\]
From Theorem~\ref{thm-w-al-v-be}  we deduce that
the embedding $W^{1,2}(\R^n,\mu)\embed L^2(B(0,1),\nu)$ is compact.
At the same time, for $q>2$ and sufficiently small $r>0$,
\[
\frac{r\nu(B(0,r))^{1/q}}{\mu(B(0,r))^{1/2}} 
\simeq \frac{r^{1/q-1/2}}{(\log(1/r))^{1/q+1/2}} \to \infty, \quad \text{as } r\to0.
\]
Proposition~\ref{prop-comp-imp-M-to0} below then 
shows that there is no compact embedding
$W^{1,2}(\R^n,\mu)\embed L^q(B(0,1),\nu)$ for any $q>2$.
In fact, $L^2(B(0,1),\nu)$ is optimal among $L^q$ spaces 
both for bounded and for compact embeddings.
Note that the exponent 2 in the conditions 
$w^{-2}\in L^1\loc(\R^2)$ and $v^2\in L^1\loc(\R^2)$ 
cannot be replaced by any larger exponent.

Obvious modifications can be done for other exponents and dimensions.
It is also possible to construct weights that are singular,  in a similar way 
at a dense set. 
\end{example}

\begin{remark}   \label{rem-nondoubl-meas}
Examples of nondoubling measures in $\R$ and $\R^2$,
with singular parts and  good enough $C_m$ for Theorem~\ref{thm-precompact}
can be found in
Bj\"orn--Bj\"orn~\cite[pp.~206--207 and Proposition~10.6]{BBnonopen}.
See also Alvarado--Haj\l asz~\cite[Example~4]{AlvHaj} for a nondoubling weighted measure on
$[0,\infty)$ that supports a $(q,p)$-Poincar\'e inequality 
(as in~\eqref{eq-(q,p)-PI}) 
precisely when $q\le p$.
At the same time, \cite[Theorem~1]{AlvHaj} shows in the
generality of metric spaces that measures supporting 
$(q,p)$-Poincar\'e inequalities with $q>p$ must be doubling.
\end{remark}

\subsection{Fractional Sobolev spaces with nonlocal gradients}
\label{sect-frac-Sob-nonlocal}

Let $\al>0$, $r_0>0$ and $\tau\ge1$ be fixed in this section.

\begin{deff}[\emph{Spaces based on Poincar\'e inequalities}]
\label{deff-al-poinc-spa}
The \emph{$E$-restricted Poincar\'e Sobolev space} $P^{\al,p}_{\tau,E}(X,\mu)$
consists  of all $u\in L^1\loc(X,\mu)$ which 
satisfy the \p-Poincar\'e inequality
\begin{equation} \label{def-PI-ineq-spc}
   \vint_{B} |u-u_{B,\mu}| \,\dmu
        \le r^\al \biggl( \vint_{\tau B} g^{p} \,\dmu \biggr)^{1/p}
\end{equation}
for some $g\in L^p(X,\mu)$
and all  balls $B=B(x,r)\subset X$ centred in $E$ and of radius
  at most $r_0$, where we implicitly assume that
$0<\mu(B)<\infty$.
The space $P^{\al,p}_{\tau,E}(X,\mu)$ is equipped with the seminorm
$\inf_g \|g\|_{L^p(X,\mu)}$, where the infimum is taken over all $g$ 
satisfying~\eqref{def-PI-ineq-spc}.
When $E=X$, we omit the subscript $E$ and write
  $P^{\al,p}_\tau$.
\end{deff}

The space $P^{\al,p}_\tau(X,\mu)$ was  for $\al=1$ (with $E=X$
  and all $0<r<\infty$ in~\eqref{def-PI-ineq-spc})
introduced in Koskela--MacManus~\cite{KoMc},
while the general case $\al>0$ was studied 
in Heikkinen--Koskela--Tuominen~\cite[Section~3]{HeiKosTuo}.
Note that 
\begin{equation}  \label{eq-P-subset-P,E}
P^{\al,p}_\tau(X,\mu)\subset P^{\al,p}_{\tau,E}(X,\mu) \quad \text{for
  all $E\subset X$}
\end{equation}
and that $P^{\al,p}_{\tau,E}(X,\mu)$ is in general \emph{not} the same as $P^{\al,p}_\tau(E,\mu)$.
If $\mu$ satisfies the doubling
  condition~\eqref{eq-def-doubl-on-E} in Assumptions~\ref{ass-doubl}
  for balls centred in $E$ and of radius at most $r_0$, 
then also (upon replacing $r_0$ in $P^{\al,p}_{\tau_2,E}$ with $r_0/\tau_2$)
\begin{equation}  \label{eq-P-tau-subset-tau}
P^{\al,p}_{\tau_1,E}(X,\mu)\subset P^{\al,p}_{\tau_2,E}(X,\mu)
\quad \text{whenever $\tau_1\le\tau_2$.}
\end{equation}

Next, we recall the definition of Haj\l asz spaces,
which were for $\al=1$ introduced by Haj\l asz~\cite{Haj96}.
Similar spaces with $\al\ne1$ were considered e.g.\ 
in~\cite{HeiKosTuo}, Koskela--Yang--Zhou~\cite{KoYaZh10}
and implicitly already in Haj\l asz--Martio~\cite{HajMar}.

\begin{deff} 
A nonnegative Borel function $g$ is a 
\emph{Haj\l asz $\al$-gradient}  of a function $u:X\to\R$ if for 
$\mu$-a.e.\ $x,y\in X$,
\begin{equation}   \label{eq-def-Haj-gr}
|u(x)-u(y)| \le d(x,y)^\al (g(x)+g(y)).
\end{equation}
The \emph{Haj\l asz space} $M^{\al,p}(X,\mu)$
consists of all $u\in L^p(X,\mu)$ for which there exists
$g\in L^p(X,\mu)$ satisfying~\eqref{eq-def-Haj-gr}.
It is equipped with the norm
\[
\| u\|_{M^{\al,p}(X,\mu)} := \| u\|_{L^p(X,\mu)} + 
   \inf \{ \| g\|_{L^p(X,\mu)} : g \text{  is as in~\eqref{eq-def-Haj-gr}} \}. 
\]
\end{deff}

By Haj\l asz~\cite[Theorems~2.1 and~2.2]{Haj-ContMat},
the spaces $P^{1,p}_\tau(\R^n,dx)$ and $M^{1,p}(\R^n,dx)$
coincide with the usual Sobolev space $W^{1,p}(\R^n)$, $p>1$.
In this case, $g$ in \eqref{eq-def-Haj-gr} can be the Hardy--Littlewood
maximal function of $|\grad u|$,
while $|\grad u|$ itself will do as $g$ in \eqref{def-PI-ineq-spc}.
The equality between the three types of spaces
is true also for sufficiently smooth bounded Euclidean domains,
cf.\ \cite[p.~196]{Haj-ContMat},
but in general, the Haj\l asz space $M^{1,p}(\Om,dx)$ can be substantially
smaller than $W^{1,p}(\Om)$, e.g.\ for the slit disc in the plane.
In fact, by Proposition~1 in Haj\l asz--Koskela--Tuominen~\cite{HajKoTuo},
for any Euclidean domain $\Om$
satisfying a measure density condition, 
the equality $M^{1,p}(\Om,dx)=W^{1,p}(\Om)$ is equivalent to $\Om$ being a
$W^{1,p}$-extension domain.

A repeated integration of \eqref{eq-def-Haj-gr} over a ball $B$ 
shows that for all $u\in M^{\al,p}(X,\mu)$,
\begin{equation}  \label{eq-p,p-PI-M}
\int_B |u(x)-u_{B,\mu}|^p\,d\mu(x) 
\le \int_B \vint_B |u(x)-u(y)|^p\,d\mu(y)\,d\mu(x)
\le Cr^{\al p} \int_B g^p\,d\mu,
\end{equation}
where $r$ is the radius of $B$. 
In particular,
\begin{equation}  \label{eq-M-subset-P}
M^{\al,p}(X,\mu) \subset P^{\al,p}_1(X,\mu) \cap L^p(X,\mu).
\end{equation}
A detailed analysis of the spaces $M^{1,p}$ and $P^{1,p}_\tau$,
as well as comparisons with other types of Sobolev spaces 
can be found in~\cite{Haj-ContMat}.

\begin{remark}[\emph{Nontriviality of $M^{\al,p}$ for $\al>1$}]   
\label{rem-walking-dim}
If $\Om\subset\Rn$ is a domain and $\al>1$, then 
Proposition~2 in Br\'ezis~\cite{Brezis}  
implies that every $u\in M^{\al,p}(\Om,dx)$  is constant. 
On the other hand, for more general sets with few rectifiable curves, $M^{\al,p}(X,\mu)$
can be nontrivial even for $\al>1$, see 
\cite[Example~6.3]{HeiKosTuo} and Hu~\cite{Hu}.
In fact,  it is easily verified  that  if 
$\hat{I}$ is the unit interval $(0,1)$ equipped with the snowflaked metric
  $\hat{d}(x,y):=|x-y|^{1/2}$, then
  $M^{2,p}(\hat{I},dx)=M^{1,p}((0,1),dx)
=W^{1,p}((0,1))$ is the usual Sobolev space when $p>1$.

Similarly,  the fractional spaces $M^{\al,p}(X,\mu)$ and $P^{\al,p}_\tau(X,\mu)$
for $0<\al<1$
coincide with the spaces $M^{1,p}(X_\al,\mu)$ and $P^{1,p}_{\tau^\al}(X_\al,\mu)$, 
respectively, where $X_\al$ denotes the space $X$ 
equipped with the snowflaked metric $d_\al(x,y):=d(x,y)^\al$.
\end{remark}

Theorem~\ref{thm-precompact} 
implies the following simple compactness results for  $P^{\al,p}_{\tau,E}$ 
and $M^{\al,p}(X,\mu)$ with very  general measures. 
Note that  $\mu$ need not be doubling.

\begin{prop}
[Embeddings for $P^{\al,p}_{\tau,E}$ with the same measure]
\label{prop-P-al-p-arb-mu}
Let $E\subset X$ be bounded Lebesgue measurable subsets of $\R^n$, 
equipped with the measure
$d\mu(x)=w(x)\,dx$, where $0<w\in L^t(X)$ for some $t\ge1$.
Then the embedding
\begin{equation*}   
P^{\al,p}_{\tau,E}(X,\mu)\cap L^1(X,\mu) \longemb L^1(E,\mu)
\end{equation*}
is compact for all $\al \ge n(1-1/p)/t$.
In particular, for the classical exponents
$p=2$ and $\al=1$, this holds for 
any weight $0<w\in L^1\loc(\R^2)$.
\end{prop}

A similar statement, with $n$ replaced by $d$, 
holds if $X$ is (a measurable subset of) a bounded 
self-similar Cantor set in $\R^n$ of Hausdorff dimension $0<d<n$, 
equipped with $d\mu=w\,d\La_d$, where $\La_d$ is the $d$-dimensional
Hausdorff measure.

\begin{proof}
For any $r_m>0$, $m=1,2,\ldots$, a simple geometrical argument shows that 
$X$ can be covered by $K_m\le C r_m^{-n}$ many
balls $B(x_i,r_m)$ with $x_i\in E$
and bounded overlap $N_m=N$ depending only on the dimension $n$.
Definition~\ref{deff-al-poinc-spa} implies that functions in 
$P^{\al,p}_{\tau,E}(X,\mu)$ satisfy the Poincar\'e type inequality
\eqref{eq-def-loc-PI} with $q=1$, $\nu=\mu$, $E_i=B(x_i,r_m)\cap X$,
$E'_i=B(x_i,\tau r_m)\cap X$, integral averages $a_{E_i}(u)=u_{E_i,\mu}$ and
\[
C_m= r_m^\al \sup_{z\in E} \mu(B(z,r_m)\cap X)^{1-1/p}.
\]
H\"older's inequality implies that
\begin{equation*}  
\mu(B(z,r_m)\cap X) \le C r_m^{n(1-1/t)} \biggl( \int_{B(z,r_m)\cap X} w(x)^t\,dx \biggr)^{1/t}
\end{equation*}
and hence
\[
C_m N_m^{1/p} K_m^{1-1/p} \le C r_m^{\al - n(1-1/p)/t}
   \sup_{z\in E} \biggl( \int_{B(z,r_m)\cap X} w(x)^t\,dx \biggr)^{(1-1/p)/t}.
\]
Absolute continuity of the integral shows that the 
last supremum tends  to zero as $r_m\to0$.
Since $\al - n(1-1/p)/t\ge0$ and the integral averages $u_{E_i,\mu}$
are bounded for each fixed $E_i$, Theorem~\ref{thm-precompact}, applied to
$E=X$, concludes the proof.
\end{proof}

Because of the inclusions~\eqref{eq-P-subset-P,E} and~\eqref{eq-M-subset-P}, 
Proposition~\ref{prop-P-al-p-arb-mu}
applies also to the Haj\l asz space $M^{\al,p}(X,\mu)$.
The next result gives a compact embedding also into $L^p$.
It partially generalizes (to $\al\ne1$) Theorem~2 
in Ka\l amajska~\cite{agnieszka}.

\begin{prop}  
[Embeddings into $L^p$ with the same  measure] 
\label{prop-q=p-M-al}
Assume that $E$ is totally bounded and that one of 
the conditions \ref{it-M(r)-X} and \ref{it-M(r)-mu}
in Lemma~\ref{lem-Hausdorff-max-princ} holds with $\la=1$,
$M_1(r) \le Cr^{-\theta}$ and $\theta\ge0$, for all balls centred in 
$E$ and of radius at most $r_0>0$. 
Assume also that $0<\mu(B)<\infty$ for all such balls.
Then the embedding 
\[
M^{\al,p}(X,\mu)\longemb L^p(E,\mu)
\]
is compact when $\alpha >\theta/p$.

Moreover, if there exists a totally bounded set $\emptyset\ne E_0\subset X$ such that
$\mu|_{E_0}$ is not a finite sum of atoms,
then $M^{\al,p}(X,\mu)\ne L^p(X,\mu)$ in this case. 
\end{prop}

\begin{proof}
Remark~\ref{rem-Hausdorff-max-princ} 
and Lemma~\ref{lem-Hausdorff-max-princ} with $\la=1$,
provide us for any $r_m>0$, $m=1,2,\ldots$,
with a covering family $\{E_i\}_{i=1}^{K_m}$, 
consisting of finitely many balls of radius $r_m\le r_0$ and with overlap
$N_m\le Cr_m^{-\theta}$. 
The statement about compactness now follows 
from~\eqref{eq-p,p-PI-M} and Theorem~\ref{thm-precompact} 
with $q=p$, $C_m=Cr_m^\al$, $G_r(u)=g$ and the integral averages 
$a_{E_i}(u)= u_{E_i,\mu}$, whose boundedness is easily justified. 

Finally, if $M^{\al,p}(X,\mu)= L^p(X,\mu)$ and $E_0\subset X$ is as in
the statement of the proposition, 
then $L^p(E_0,\mu)$ is continuously embedded in $L^p(X,\mu)$,
by means of the zero extension, and hence compactly embedded in
itself, because of the first part of the proposition.
This implies that  $L^p(E_0,\mu)$  must be
finite-dimensional, by the Riesz lemma 
\cite[Theorem~2.5-5]{kreysz}, i.e.\ $\mu|_{E_0}$ is a finite sum of atoms.
\end{proof}

\section{Derivation of Poincar\'e type inequalities}

\label{sect-derive-PI}

Next, we prove the Poincar\'e
type inequalities~\eqref{eq-def-loc-PI} and~\eqref{eq-def-loc-PI-bdd}  
for balls
in spaces with a good domain measure~$\mu$.
This generalizes Theorem~7 in Bj\"orn~\cite{FennAnn} to measures restricted 
to certain (possibly lower-dimensional) subsets.
The proof has been 
inspired by Haj\l asz--Koskela~\cite[Theorem~5.3]{HaKobook}
and Heinonen--Koskela~\cite[Lemma~5.15]{HeKo98}.

\begin{thm}[Self-improvement of Poincar\'e inequalities]    
\label{thm-two-weight-PI-eps}
Let $1\le p<q<\infty$, $r_0>0$,  $\al>0$ and $\la\ge1$.
Assume that the measures $\mu$ and $\nu$ satisfy the doubling conditions~\eqref{eq-def-doubl-on-E} and
  \eqref{eq-dim-est-sig-E} in
Assumptions~\ref{ass-doubl} for all balls centred in $E$ and with radius
at most~$r_0$.
Let $u\in L^1\loc(X,\mu)$ be such that  
\begin{equation}   \label{eq-def-Leb-pt}
u(x)=\lim_{r\to0} \vint_{B(x,r)} u\,d\mu
\quad \text{for $\nu$-a.e.\  $x\in E$}
\end{equation}
and, for some function $g\in L^p\loc(X,\mu)$,
the \p-Poincar\'e  inequality~\eqref{eq-def-abstr-PI}
holds with dilation $\la\ge1$ on all balls $B$ centred in $E$ and with radius at most $r_0$.

If the local Poincar\'e constant on $E$,
\begin{equation}  \label{eq-def-sup-Theta}
\Theta_{q,\la}(r):=\sup_{0<\rho\le r}
\sup_{x\in E} \frac{\rho^\al \nu(B(x,\rho))^{1/q}}{\mu(B(x,\la \rho))^{1/p}}, 
\quad \text{where }0<r\le r_0,
\end{equation}
satisfies $\Theta_{q,\la}(r_0)<\infty$,
then the following are true with $C_{q,\la}>0$ 
depending only on $p$, $q$, $\mu$, $\nu$, $\al$, $\la$ and $\de$, but not on 
$x$, $r$, $u$ and~$g$\/{\rm:}
\begin{enumerate}
\renewcommand{\theenumi}{\textup{(\roman{enumi})}}%
\renewcommand{\labelenumi}{\theenumi}
\item For all $1\le q'<q$, the following two-weighted 
Poincar\'e type inequality  
\begin{equation}
\biggl( \int_{B\cap E} |u-u_{B,\mu}|^{q'} \,d\nu \biggr)^{1/q'}
   \le C_{q,\la} \Theta_{q,\la}(r) \frac{\nu(B\cap E)^{1/q'-1/q}}{q-q'}
             \biggl( \int_{2\la B} g^p \,d\mu \biggr)^{1/p}
\label{eq-two-weight-PI}
\end{equation}
holds for all balls $B=B(x,r)$ with $x\in E$ and $10\la r\le r_0$.
In particular, $u\in L\loc^{q'}(E,\nu)$.
\item \label{it-trunc-prop}
If the pair $(u,g)$  also satisfies the truncation property,
i.e.\ for all $l<k$, the inequality~\eqref{eq-def-abstr-PI} 
holds with $u$ and $g$ replaced by
\begin{equation}  \label{eq-def-trunc}
u_{l,k}:= \max\{l,\min \{u,k\}\} \quad \text{and} \quad
g_{l,k}:= g\chi_{\{l<u<k\}},
\end{equation}
respectively, then \eqref{eq-two-weight-PI} holds also for $q'=q$.
\end{enumerate}
\end{thm}

A straightforward application of the triangle inequality shows that 
the integral average
$u_{B,\mu}$ in the left-hand side of \eqref{eq-two-weight-PI} can always
be replaced by $u_{B\cap E,\nu}$ at the cost of an additional factor $2$
on the right-hand side.

The truncation property in 
Theorem~\ref{thm-two-weight-PI-eps}\,\ref{it-trunc-prop} 
is satisfied if $g=|\grad u|$
in $\R^n$ or if $g=g_u$ is the minimal \p-weak upper gradient of $u$,
but not for the nonlocal Haj\l asz gradients,
both in $\R^n$ and in metric spaces. 
Before proving Theorem~\ref{thm-two-weight-PI-eps}, 
we formulate some remarks about its assumptions.

\begin{remark}
[\emph{Choice of $\al$ and $q$}]
In \eqref{eq-def-abstr-PI} and \eqref{eq-def-sup-Theta} 
we allow also  $\al>1$.
In the classical \p-Poincar\'e inequality in Euclidean
spaces, such a choice would force $u$ to be constant,
by Heikkinen--Koskela--Tuominen~\cite[Corollary~1.2]{HeiKosTuo}.
On the other hand, for 
functions defined on fractal sets, 
even $\al>1$ can give nontrivial results,
cf.\  Remark~\ref{rem-walking-dim}.

The proof of Theorem~\ref{thm-two-weight-PI-eps} requires $q>p$, 
but once \eqref{eq-two-weight-PI} has been
proved for some $q>p$, a similar inequality holds also for smaller exponents,
by H\"older's inequality. 
However, even in that case, \eqref{eq-def-sup-Theta} 
has to be assumed with some initial $q>p$ for the proof to apply. 
It would be interesting to see which Poincar\'e inequalities can be 
obtained from \eqref{eq-def-abstr-PI} when $\Theta_{q,\la}(r_0)<\infty$ 
only for some $q\le p$.
\end{remark}

\begin{remark}[\emph{Role of $\de$ and uniform perfectness}] 
\label{rem-doubl-imp-dim}
The size of $\de$ in~\eqref{eq-dim-est-sig-E} of Assumptions~\ref{ass-doubl}   
is unimportant and only has effect 
on the constant $C_{q,\la}$ in the Poincar\'e inequality~\eqref{eq-two-weight-PI}.
It can be proved as in \cite[Corollary~3.8]{BBbook}
that \eqref{eq-dim-est-sig-E}
follows for some $\de>0$ 
from the local doubling condition~\eqref{eq-def-doubl-on-E} 
if $E$ is connected or, more generally, locally uniformly perfect, 
i.e.\ there is $0<a<1$ such that the sets
$(E\cap B(x,r))\setm B(x,ar)$ are nonempty whenever 
$x\in E$ and $0<r\le r_0$.
Note that many fractal sets, 
which are natural candidates for $E$,
are uniformly perfect but not connected.
If $\nu$ is supported on $E$ and
satisfies \eqref{eq-def-doubl-on-E}, 
then \eqref{eq-dim-est-sig-E} with some $\de >0$ is equivalent
to the local uniform perfectness of~$E$, 
see Mart\'\i n--Ortiz~\cite[Lemma~7]{MarOrt}.
\end{remark}

\begin{remark}[\emph{$\mu$-Lebesgue points}]   
\label{rem-Leb-pts}
Assumption \eqref{eq-def-Leb-pt} in
Theorem~\ref{thm-two-weight-PI-eps} 
is not too restrictive for our applications.
Since $u\in L^1\loc(X,\mu)$ and $\mu$
satisfies the local doubling condition~\eqref{eq-def-doubl-on-E} on $E$,
it can be shown as in
Heinonen~\cite[Theorems~1.6 and 1.8]{Heinonen}, using only balls centred in $E$,
that $\mu$-a.e.\ $x\in E$ satisfies the equality in~\eqref{eq-def-Leb-pt}.

Lemma~\ref{lem-est-Leb-pts} below then shows that 
an $L^1\loc(X,\mu)$-\emph{re\-pre\-sent\-ative} of $u$ in
Theorem~\ref{thm-two-weight-PI-eps}
always has $\mu$-Lebesgue points $\nu$-a.e.\ in $E$.
\end{remark}

\begin{lem}[Non-Lebesgue points]   \label{lem-est-Leb-pts}
Assume that the measures $\mu$ and $\nu$ satisfy
\eqref{eq-def-doubl-on-E} and \eqref{eq-dim-est-sig-E} in Assumptions~\ref{ass-doubl}.
Let $u\in L^1\loc(X,\mu)$ be such that 
the Poincar\'e type inequality~\eqref{eq-def-abstr-PI} 
with parameters $\al>0$ and $\la\ge1$ holds 
for some function $g\in L^p\loc(X,\mu)$
and all balls centred in $E$ and of radius at most $r_0>0$.
Assume that 
\begin{equation*}   
\sup_{0<r\le r_0} \frac{r^{\al p} \nu(B(x,r))^{1/q}}{\mu(B(x,\la r))^{1/p}} < \infty
\quad \text{for every $x\in E$,}
\end{equation*}
where $q>p$.  Then $\nu(E_0)=0$, where
\[
E_0:=\{x\in E: \text{the limit in \eqref{eq-def-Leb-pt} does not exist  
or is not finite}\}.
\]
In particular, a representative $\ub$ of $u$ in  $L^1\loc(X,\mu)$
has $\mu$-Lebesgue points $\nu$-a.e.\ in $E$ 
in the sense of~\eqref{eq-def-Leb-pt}.
\end{lem}

\begin{proof}
By splitting $E$ into subsets and using the countable subadditivity of
$\nu$, we can assume that  
$E$ is contained in a ball $B_0$ of radius $r_0$ and that for some $M$,
\begin{equation}   \label{eq-Theta-with-M}
\sup_{0<r\le r_0} \frac{r^{\al p} \nu(B(x,r))^{1/q}}{\mu(B(x,\la r))^{1/p}} \le M
\quad \text{for all $x\in E$}.
\end{equation}
The Poincar\'e type inequality~\eqref{eq-def-abstr-PI} 
and the local doubling property~\eqref{eq-def-doubl-on-E} of $\mu$
imply as in \cite[(5.1) and the proof of Theorem~5.1]{BBbook} that 
\begin{equation}    \label{eq-ub(x)-ex}
\lim_{r\to0} \vint_{B(x,r)}u\,d\mu \quad \text{exists 
   and is finite,}
\end{equation}
whenever $x\in E$ is such that the fractional maximal function
\[
\sup_{0<r\le r_0/5\la} r^{\al'} \biggl( \vint_{B(x,r)} g^p\,d\mu \biggr)^{1/p} 
\]
is finite for some $0<\al'<\al$.
It therefore suffices to estimate $\nu(E_t)$ for $t>0$, where
\[
E_t = \biggl\{x\in E: \sup_{0<r\le r_0/5\la} r^{\al' p} 
    \vint_{B(x,r)} g^p\,d\mu > t\biggr\}.
\]
Let $\al'=\al-\de(1/p-1/q)$, where $\de$ is as in \eqref{eq-dim-est-sig-E}.
For each  $x\in E_t$, find $0<r_x\le r_0/5\la$ such that 
\[
\mu(B(x,r_x)) < \frac{r_x^{\al' p}}{t} \int_{B(x,r_x)} g^p\,d\mu.
\]
Together with the local doubling property~\eqref{eq-def-doubl-on-E} of $\mu$  
and \eqref{eq-Theta-with-M}, this shows that
\[
\nu(B(x,r_x))^{p/q} 
  \le \frac{M\mu(B(x,\la r_x))}{r_x^{\al p}} 
  \le \frac{C M r_x^{(\al'-\al)p}}{t} \int_{B(x,r_x)} g^p\,d\mu.
\]
Moreover, \eqref{eq-dim-est-sig-E} yields
\[
\nu(B(x,r_x))^{1-p/q} \le C \nu(B(x,r_0))^{1-p/q}  \Bigl( \frac{r_x}{r_0} \Bigr)^{\de(1-p/q)}.
\]
Combining the last two estimates with the local doubling
property~\eqref{eq-def-doubl-on-E} for $\nu$ and the fact that
$(\al-\al')p = \de(1-p/q)$, we get
\[
\nu(B(x,5r_x)) \le   \frac{C}{t} \int_{B(x,r_x)} g^p\,d\mu,
\]
where $C$ depends on $M$, $r_0$ and $B_0$ but not on $x$ and $r_x$.
The balls $B(x,r_x)$, with $x\in E_t$, cover $E_t$ and the 
Vitali type 5-covering lemma \cite[Lemma~5.5]{HeKo98}
implies that there are at most countably many pairwise disjoint balls 
$B_j=B(x_j,r_{x_j})$ with $x_j\in E_t$ such that $E_t\subset \bigcup_j 5B_j$.
We thus get that
\[
\nu(E_t) \le \sum_j \nu(5B_j) \le \frac{C}{t}  
         \sum_j \int_{B_j} g^p\,d\mu
\le \frac{C}{t}  \int_{\bigcup_j B_j} g^p\,d\mu.
\]
Since $g\in L^p\loc(X,\mu)$ and $E_0\subset E_t$ for all $t>0$,
letting $t\to\infty$ shows that $\nu(E_0)=0$.
Redefining $u$ on $E\setm E_0$ by \eqref{eq-ub(x)-ex} then provides the desired
$L^1\loc(X,\mu)$-representative, see Remark~\ref{rem-Leb-pts}.
\end{proof}

\subsection{Proof of Theorem~\ref{thm-two-weight-PI-eps}}

\begin{proof} 
Let $x_0\in E$ and $B=B(x_0,r)$ with $10\la r\le r_0$ be fixed. 
Let $u$ and $g$ be as in the statement of the theorem.
Assume that $x\in B\cap E$ is a Lebesgue point of $u$ with respect to $\mu$
and set
\begin{equation*} 
B_0=2B \quad \text{and} \quad B_j=B(x,2^{1-j} r),
\quad j=1,2,\ldots.
\end{equation*}
The Poincar\'e inequality~\eqref{eq-def-abstr-PI}, a telescoping argument
and \eqref{eq-def-doubl-on-E} for $\mu$ imply that
\begin{align*}
|u(x)-u_{B_0,\mu}| &= \lim_{j\to\infty} |u_{B_j,\mu} - u_{B_0,\mu}|
  \le \sum_{j=0}^{\infty} \vint_{B_{j+1}} |u-u_{B_j,\mu}|\,d\mu \nonumber\\
   &\le C \sum_{j=0}^{\infty} \vint_{B_{j}} |u-u_{B_j,\mu}|\,d\mu
  \le C \sum_{j=0}^{\infty} (2^{-j} r)^\al \biggl( \vint_{\la B_j} g^p \,\dmu
                        \biggr)^{1/p}
\end{align*}
and
\[
|u_{B,\mu}-u_{B_0,\mu}|
        \le C r^\al \biggl( \vint_{\la B_0} g^p \,\dmu \biggr)^{1/p}.
\]
Applying~\eqref{eq-def-sup-Theta} to the balls $B_j$
and $B_0$ yields
\begin{equation}   \label{eq-Sig'+Sig''}
|u(x)-u_{B,\mu}| \le C \Theta_{q,\la}(r) 
      \sum_{j=0}^{\infty} \frac{1}{\nu(B_j)^{1/q}}
       \biggl( \int_{\la B_j} g^p \,\dmu \biggr)^{1/p}.
\end{equation}
Write the above sum as $\Sigma' + \Sigma''$, where the summation
in $\Sigma'$ and  $\Sigma''$ is over $j<j_0$ and $j\ge j_0$, respectively,
and $j_0$ will be chosen later.
By the local dimension condition~\eqref{eq-dim-est-sig-E} 
for $\nu$ we have
\begin{align}
\nu(B_j) &\ge C 2^{(j_0-j)\de} \nu(B_{j_0}) \quad \text{for }j<j_0,
\nonumber \\
\nu(B_j) &\le C 2^{(j_0-j)\de} \nu(B_{j_0}) \quad \text{for }j\ge j_0,
\label{eq-est-nu-Bj0}
\end{align}
and hence,
\begin{align*}
\Sigma' &= \sum_{j=0}^{j_0-1} \frac{1}{\nu(B_j)^{1/q}}
                        \biggl( \int_{\la B_j} g^p \,\dmu \biggr)^{1/p} \\
& \le C \sum_{j=0}^{j_0-1} \frac{2^{(j-j_0){\de}/q}}{\nu(B_{j_0})^{1/q}}
                        \biggl( \int_{\la B_j} g^p \,\dmu \biggr)^{1/p} 
              \le  \frac{C}{\nu(B_{ j_0})^{1/q}}
                        \biggl( \int_{\la B_0} g^p \,\dmu \biggr)^{1/p}.
\end{align*}
At the same time, because of~\eqref{eq-est-nu-Bj0} 
and since $1/p-1/q>0$, we have
\[
 \Sigma'' = \sum_{j=j_0}^{\infty} \nu(B_j)^{1/p-1/q}
  \biggl( \frac{1}{\nu(B_j)} \int_{\la B_j} g^p \,\dmu \biggr)^{1/p}
                \le C \nu(B_{j_0})^{1/p-1/q} M(x)^{1/p},
\]
where
\begin{equation}   \label{eq-def-M(x)}
        M(x) := \sup_{B'} \frac{1}{\nu(B')} \int_{\la B'} g^p \,\dmu
\end{equation}
is a generalized maximal function,
with the supremum taken over all balls $B'$,
containing $x$ and of radius $r'\le2r$.
Inserting this into~\eqref{eq-Sig'+Sig''} yields
\begin{align}   
|u(x)-u_{B,\mu}| &\le C \Theta_{q,\la}(r) (\Sigma' + \Sigma'')
\label{eq-Sigma'+Sigma''} \\
&\le C \frac{\Theta(r)}{\nu(B_{j_0})^{1/q}}
     \biggl( \biggl(\int_{\la B_0} g^p \,\dmu \biggr)^{1/p}
      + \bigl( \nu(B_{j_0})  M(x) \bigr)^{1/p} \biggr). \nonumber
\end{align}
To minimize the right-hand side of~\eqref{eq-Sigma'+Sigma''}, 
we will choose $j_0$ so that 
$\Sigma^{'}$ and $\Sigma^{''}$ are comparable.
We can assume that $M(x)>0$, as otherwise \eqref{eq-def-M(x)},
\eqref{eq-def-abstr-PI} and \eqref{eq-def-Leb-pt}
imply that $g\equiv0$ $\mu$-a.e.\ in $\la B$ and 
$u\equiv u_{B,\mu}$ both $\mu$-a.e.\ and $\nu$-a.e.\ in $B$, 
so that the statement of the theorem holds trivially. 

A standard argument using the Vitali type 5-covering lemma \cite[Lemma~5.5]{HeKo98},
together with the doubling property \eqref{eq-def-doubl-on-E} of
$\nu$ on $E$ and the fact that
$10\la r\le r_0$, shows 
for all $\tau>0$ the weak type inequality
\begin{equation}   \label{eq-Vitali-M(x)}
\nu( \{x\in B\cap E: M(x)\ge \tau \}) 
\le \frac{C}{\tau} \int_{2\la B} g^p \,\dmu,
\end{equation}
cf.\ Chapter~1 in Stein~\cite{Stein}.
Hence $M(x)<\infty$ for $\nu$-a.e.\ $x\in B\cap E$.
We consider only such $x$ in the rest of the proof.
Since $B'=B_0$ is allowed in~\eqref{eq-def-M(x)} and
$\nu(B_j)\to0$, as $j\to\infty$, 
we can use the doubling property of $\nu$ to
find $j_0=j_0(x,r)\ge 0$ such that
\[
\frac{1}{\nu(B_{j_0})} \int_{\la B_0} g^p \,\dmu \le
M(x) \le \frac{1}{\nu(B_{j_0+1})} \int_{\la B_0} g^p \,\dmu
\le \frac{C}{\nu(B_{j_0})} \int_{\la B_0} g^p \,\dmu.
\]
Then $\nu(B_{j_0})$ is comparable to 
$M(x)^{-1} \int_{\la B_0} g^p \,\dmu$ and 
inserting this into \eqref{eq-Sigma'+Sigma''} yields
\begin{equation*}   
|u(x)-u_{B,\mu}| \le C \Theta_{q,\la}(r)
         \biggl( \int_{2\la B} g^p \,\dmu \biggr)^{1/p-1/q} M(x)^{1/q},
\end{equation*}
since $\la B_0=2\la B$.
Using~\eqref{eq-Vitali-M(x)} we therefore conclude that
\begin{equation}   \label{eq-est-nu-level-g-u-mu}
t^q \nu(\{ x\in B\cap E: |u(x)-u_{B,\mu}| \ge t \})
         \le C [\Theta_{q,\la}(r)]^q
                        \biggl( \int_{2\la B} g^p \,\dmu \biggr)^{q/p}.
\end{equation}
Now, for $1\le q'<q$, we have by the Cavalieri principle that
\[
\int_{B\cap E} |u-u_{B,\mu}|^{q'}\,d\nu 
= q' \int_0^\infty t^{q'-1} \nu(\{ x\in B\cap E: |u(x)-u_{B,\mu}| \ge t\}) \,dt.
\]
Splitting the integral as $\int_0^\infty = \int_0^{t_0} + \int_{t_0}^\infty$
and estimating the integrands by $\nu(B\cap E)$ and
\eqref{eq-est-nu-level-g-u-mu}, respectively, yields
\[
\int_{B\cap E} |u-u_{B,\mu}|^{q'}\,d\nu 
\le t_0^{q'} \nu(B\cap E) 
    + \frac{Cq' t_0^{q'-q} [\Theta_{q,\la}(r)]^q}{q-q'} \biggl( \int_{2\la B} g^p \,\dmu \biggr)^{q/p}.
\]
Choosing
\[
t_0 = \frac{\Theta_{q,\la}(r)}{\nu(B\cap E)^{1/q}} 
              \biggl( \int_{2\la B} g^p \,\dmu \biggr)^{1/p}
\]
then gives~\eqref{eq-two-weight-PI} for $q'<q$.

To obtain~\eqref{eq-two-weight-PI}
also for $q'=q$, we proceed by Maz\cprime ya's truncation method \cite{Maz}.
We shall estimate the integral
\[
\vint_{B\cap E} v^{q}\,d\nu, \quad \text{where } v:=\max\{u-u_{B,\mu},0\}.
\]
We can clearly assume that $v\not\equiv 0$.
Let $k_0$ be the largest integer such that
\begin{equation}
2^{k_0-1} \le \vint_{B} v\,d\mu,
               \label{choose-k0}
\end{equation}
and for $k> k_0$ let $u_k:= u_{l,l+2^k}$ be the truncations of $u$
at levels $l:=u_{B,\mu}+2^k$ and $l+2^k=u_{B,\mu}+2^{k+1}$,
defined as in \eqref{eq-def-trunc}.
Note that $u_k =(u-u_{B,\mu})_{2^k,2^{k+1}} + u_{B,\mu}\le v+l$.
Since \eqref{choose-k0} fails for $k>k_0$, we have for such $k$,
\[
(u_k)_{B,\mu} \le l + \vint_{B} v \,d\mu < l + 2^{k-1}.
\]
Moreover, $v(x)\ge 2^{k+1}$ if and only if $u_k(x)\ge l+2^k$.
Hence the
estimate~\eqref{eq-est-nu-level-g-u-mu}, applied with $t=2^{k-1}$ to the
truncations $u_k$, yields that for $k>k_0$,
\begin{align*}
&2^{(k-1)q} \nu(\{ x\in B\cap E: v(x)\ge 2^{k+1} \}) \\
 & \quad \quad \quad \quad
\le 2^{(k-1)q} \nu(\{ x\in B\cap E: |u_k(x)-(u_k)_{{B,\mu}}| \ge 2^{k-1} \}) \\
 & \quad \quad \quad \quad
\le C [\Theta_{q,\la}(r)]^q \biggl( \int_{2\la B} g_{k}^p \,\dmu \biggr)^{q/p},
\end{align*}
where $g_k:= g\chi_{\{2^k<v<2^{k+1}\}}$.
Summing over $k> k_0$ we get that
\begin{align}
\int_{\{x\in B\cap E: v(x)>2^{k_0+2}\}} v^{q} \,d\nu
      &\le \sum_{k=k_0+1}^{\infty} 2^{(k+2)q}
            \nu(\{ x\in B\cap E: 2^{k+2} \ge v(x)\ge 2^{k+1} \}) \nonumber \\
      &\le C [\Theta_{q,\la}(r)]^q \sum_{k=k_0+1}^{\infty}
             \biggl( \int_{2\la B} g_{k}^p \,\dmu \biggr)^{q/p} \nonumber\\
      &\le C [\Theta_{q,\la}(r)]^q \biggl( \int_{2\la B} g^p \,\dmu \biggr)^{q/p},
\label{est-int-A-large-k}
\end{align}
where in the last step we used the elementary inequality
\(
\sum_{j=1}^\infty a_j^{q/p} 
\le  \Bigl(\sum_{j=1}^\infty a_j\Bigr)^{q/p},
\)
valid since $q>p$.
At the same time, \eqref{choose-k0} together with the fact that
$0\le v \le |u-u_{B,\mu}|$ and the assumed Poincar\'e type 
inequality~\eqref{eq-def-abstr-PI} implies that
\begin{align*}
\int_{\{x\in B\cap E:v(x)\le2^{k_0+2}\}} v^{q} \,d\nu
      &\le 2^{(k_0+2)q} \nu(B\cap E)
   \le C \nu(B\cap E) \biggl( \vint_B v\,d\nu \biggr)^{q}  \nonumber   \\
&\le C\frac{r^{\al q} \nu(B\cap E)}{\mu(\la B)^{q/p}}
          \biggl( \int_{\la B} g^p \,\dmu \biggr)^{q/p}.
\end{align*}
Adding this to \eqref{est-int-A-large-k}
we conclude, using the definition of $\Theta_{q,\la}(r)$, that
\begin{equation*}
\int_{B\cap E} \max\{u-u_{B,\mu},0\}^{q}\,d\nu
\le C[\Theta_{q,\la}(r)]^q \biggl( \int_{2\la B} g^p \,\dmu \biggr)^{q/p}.
\end{equation*}
The integral $\int_{B\cap E} \max\{u_{B,\mu}-u,0\}^{q}\,d\nu$ 
can be estimated similarly,    
since \eqref{eq-def-abstr-PI} and the truncation property remain invariant  
when replacing the pair $(u,g)$ by $(-u,g)$. 
Hence, we can conclude that \eqref{eq-two-weight-PI} holds
and $u\in L^q(B\cap E,\nu)$.
\end{proof}

\begin{proof}[Proof of Theorem~\ref{thm-intro-PI}]
Since $E$ is uniformly perfect, Remark~\ref{rem-doubl-imp-dim}
implies that the dimension condition~\eqref{eq-dim-est-sig-E} holds for $\nu$. 
For a fixed ball $B\subset X$ centred in $E$, the statement now follows by applying
Theorem~\ref{thm-two-weight-PI-eps} with $E$ replaced by $E\cap B$.
Note that the assumptions in Theorem~\ref{thm-two-weight-PI-eps} 
are satisfied for any $0<r_0<\infty$.
\end{proof}

\section{Embeddings with locally doubling measures}
\label{sect-loc-doubl-meas}

Theorems~\ref{thm-precompact} and~\ref{thm-two-weight-PI-eps}
imply the following practical conditions for compact and bounded embeddings. 
Recall from the introduction that an embedding into $L^q(E,\nu)$
means that $\ub\in L^q(E,\nu)$, where $\ub$ is given by~\eqref{eq-Leb-repr}.
Conditions~\eqref{eq-def-doubl-on-E} and
  \eqref{eq-dim-est-sig-E} in
Assumptions~\ref{ass-doubl} and $\Theta_{q,\la}(r_0)<\infty$ imply
that $\ub=u$ $\mu$-a.e.\ in $E$, see Remark~\ref{rem-Leb-pts}.
Recall Definition~\ref{deff-al-poinc-spa} of the $E$-restricted
Poincar\'e Sobolev space $P^{\al,p}_{\tau,E}(X,\mu)$.

\begin{thm}[Embeddings with locally doubling measures]   
\label{thm-comp-for-q'<q-gen}
Assume that $E$ is totally bounded and that the measures $\mu$ and $\nu$
satisfy the doubling conditions~\eqref{eq-def-doubl-on-E} and
\eqref{eq-dim-est-sig-E} 
in Assumptions~\ref{ass-doubl}
for all balls $B$ centred in $E$ and with radius at most~$r_0$.
Let $\Y\subset P^{\al,p}_{\la,E}(X,\mu)$, $\la\ge1$,
be equipped with one of the norms 
\begin{equation}   \label{eq-choose-norm-gen}
\|u\|_{L^1(E,\nu)} + \|G(u)\|_{L^p(X,\mu)} \quad \text{or} \quad
\|u\|_{L^1(X,\mu)} + \|G(u)\|_{L^p(X,\mu)},
\end{equation}
where $G(u)$ is such that~\eqref{def-PI-ineq-spc} in
Definition~\ref{deff-al-poinc-spa} holds with $g=G(u)$.

Then the following hold for $q>p$\/{\rm:}
\begin{enumerate}
\renewcommand{\theenumi}{\textup{(\roman{enumi})}}%
\renewcommand{\labelenumi}{\theenumi}
\item \label{it-Th<infty-comp}
If $\Theta_{q,\la}(r_0)<\infty$ then the embedding
$\Y \embed L^{q'}(E,\nu)$ is 
compact for all $q'<q$.
\item \label{it-Th-trunc}
If functions in $\Y$ satisfy the truncation property 
in Theorem~\ref{thm-two-weight-PI-eps}\,\ref{it-trunc-prop}, then the
embedding $\Y \embed L^{q}(E,\nu)$ 
is bounded when $\Theta_{q,\la}(r_0)<\infty$ and compact when $\Theta_{q,\la}(r)\to0$, as $r\to0$. 
\end{enumerate}
\end{thm}

Since $E$ is totally bounded, we have $\nu(E)<\infty$, and so embeddings into
$L^q$ for $q\le p$ follow from those with $q>p$. 
Propositions~\ref{prop-comp-imp-M-to0} and~\ref{prop-Th<infty-necess} below show that 
statement~\ref{it-Th-trunc} admits a converse, at least for some choices of $\Y$.
See also Example~\ref{ex-optimal-emb} for a situation where 
\ref{it-Th<infty-comp} is not strong 
enough to include the optimal exponent $q=2$ for a compact embedding.
Total boundedness of $E$ is natural since e.g.\ the Sobolev space $W^{1,p}(\R^n)$ does not embed
compactly into any $L^q(\R^n)$.
By Proposition~\ref{prop-tot-bdd}, it is also necessary 
at least for some $\Y$ when $q=p$, $E=X$ and $\nu=\mu$ is doubling.

To estimate $\Theta_{q,\la}(r)$, we can use the dimension
conditions~\eqref{eq-dim-mu-only-s} and~\eqref{eq-dim-nu-only-sig}
from Assumptions~\ref{ass-doubl}.
Note that \eqref{eq-dim-nu-only-sig} with $\sig=\de$ follows from
\eqref{eq-dim-est-sig-E} 
if $E$ can be covered by finitely many balls $\{B^0_j\}_j$ 
of radius $r_0$ and $\nu(2B^0_j)<\infty$.
Indeed, for each $x\in E$ there is $B^0_j$ such that $x\in B^0_j$ and
hence for $r\le r_0$,
\begin{equation}  \label{eq-B-r/r0}
\nu(B(x,r)) \le C \nu(B(x,r_0)) \Bigl( \frac{r}{r_0} \Bigr)^{\de}
\le \frac{C \max_j\nu(2B^0_j)}{r_0^{\de}} r^\de.
\end{equation}
Similarly, an iteration of \eqref{eq-def-doubl-on-E} 
implies \eqref{eq-dim-mu-only-s} with some $s>0$, see~\cite[Lemma~3.3]{BBbook}.

\begin{proof}[Proof of Theorem~\ref{thm-comp-for-q'<q-gen}]
Choose a sequence $r_m\to0$ so that $0<r_m\le r_0/10\la$. 
As in Remark~\ref{rem-Hausdorff-max-princ}, and using the total
boundedness of $E$,
the set $E$ can be for each fixed $m=1,2,\ldots$, covered by 
finitely many balls $B_i$ with centres in $E$ and radius $r_m$ 
so that the balls $\tfrac12 B_i$ are pairwise disjoint.
This gives us a covering family with $E_i=B_i\cap E$ and  $E'_i=2\la B_i$.
Lemma~\ref{lem-Hausdorff-max-princ} with $X$ replaced by $E$,
together with the local doubling property on $E$, shows that the 
dilated balls $2\la B_i$ have a bounded overlap~$N$ independent of~$m$.
Note that $\nu(E_i)\le Cr_m^\de$ with $C$ independent of $m$ and $i$,
by~\eqref{eq-B-r/r0}.
By Remark~\ref{rem-Leb-pts}, 
functions in $\Y$ (modified by~\eqref{eq-Leb-repr}) satisfy~\eqref{eq-def-Leb-pt}.

To prove \ref{it-Th<infty-comp},
consider a sequence in $\Y$, which is bounded with respect to
one of the norms in \eqref{eq-choose-norm-gen}.
Theorem~\ref{thm-two-weight-PI-eps} implies that 
\eqref{eq-two-weight-PI}  holds with $G_m(u)=G(u)$ for all $1\le q'<q$. 
Thus, Assumptions~\ref{ass-comp}  are satisfied for all such $q'$ with
\begin{equation}   \label{eq-pick-C_m}
C_m:= C\Theta_{q,\la}(r_m) \max_i \nu(E_i)^{1/q'-1/q}
  \le C\Theta_{q,\la}(r_0) r_m^{(1/q'-1/q)\de}
\end{equation}
and with both integral
averages $a_{E_i}(u)=u_{B_i,\mu}$ and $a_{E_i}(u)=u_{E_i,\nu}$.
These averages are clearly bounded for each fixed $E_i$ and each of the
norms in~\eqref{eq-choose-norm-gen}. 
Since $\Theta_{q,\la}(r_0)<\infty$,  
we have $C_m\to0$ and an application of Theorem~\ref{thm-precompact}
proves~\ref{it-Th<infty-comp}.

If  the truncation property holds in~\ref{it-Th-trunc}, then \eqref{eq-two-weight-PI}
and \eqref{eq-pick-C_m} hold also for $q'=q$ with $C_m:= C\Theta_{q,\la}(r_m)$.
Boundedness then follows from Theorem~\ref{thm-bounded} since 
$\Theta_{q,\la}(r_0)<\infty$ implies that Assumptions~\ref{ass-bdd} 
are satisfied with the above integral averages $a_{E_i}(u)$.
Similarly, $\Theta_{q,\la}(r_m)\to0$ in~\ref{it-Th-trunc}
implies that Assumptions~\ref{ass-comp}  
hold with $C_m\to0$, which 
gives a compact embedding also into $L^{q}(E,\nu)$.
\end{proof}

\begin{cor}[Embeddings with dimension conditions]   
\label{cor-comp-dim}
Assume that $E$ is totally bounded and that $\mu$ and $\nu$
satisfy the doubling and dimension conditions~\eqref{eq-def-doubl-on-E}--\eqref{eq-dim-nu-only-sig} 
in Assumptions~\ref{ass-doubl} with exponents $s$ and $\sig>s-\al p$,
for all balls $B$ centred in $E$ and with radius at most~$r_0$.
Let $\Y\subset P^{\al,p}_{\la,E}(X,\mu)$, $\la\ge1$, 
be equipped with one of the norms~\eqref{eq-choose-norm-gen}.

Then the following hold for $q>p$\/{\rm:}
\begin{enumerate}
\renewcommand{\theenumi}{\textup{(\roman{enumi})}}%
\renewcommand{\labelenumi}{\theenumi}
\item \label{it-dim-comp}
The embedding 
$\Y \embed L^{q}(E,\nu)$ is compact whenever $q(s-\al p)<\sig p$.
When $\nu=\mu$, condition~\eqref{eq-dim-nu-only-sig} is not needed and
the embedding is  compact whenever $q(s-\al p)< sp$.
\item \label{it-dim-trunc}
If functions in $\Y$ satisfy the truncation property 
in Theorem~\ref{thm-two-weight-PI-eps}\,\ref{it-trunc-prop}, then the
embedding $\Y \embed L^{q}(E,\nu)$ is bounded whenever $q(s-\al p)\le\sig p$.
When $\nu=\mu$, condition~\eqref{eq-dim-nu-only-sig} is not needed and
the embedding is  bounded whenever $q(s-\al p)\le sp$.
\end{enumerate}
\end{cor}

\begin{proof}
For \ref{it-dim-comp}, choose $\qb>\max\{q,p\}$ such that $\qb(s-\al p)<\sig p$.
Consider the local Poincar\'e constant $\Theta_{\qb,\la}(r)$, defined 
as in~\eqref{eq-def-sup-Theta} with $q$ replaced by~$\qb$.
Using \eqref{eq-dim-mu-only-s} and \eqref{eq-dim-nu-only-sig}
we see that $\Theta_{\qb,\tau}(r_0)<\infty$.
Theorem~\ref{thm-comp-for-q'<q-gen}~\ref{it-Th<infty-comp}, with $q$ and $q'$ 
replaced by~$\qb$ and~$q$, 
then concludes the proof in the general case $\nu\ne\mu$.
The special case $\nu=\mu$ is similar and uses only \eqref{eq-dim-mu-only-s}.

Part~\ref{it-dim-trunc}
follows directly from Theorem~\ref{thm-comp-for-q'<q-gen}\,\ref{it-Th-trunc}
and the fact that the dimension conditions~\eqref{eq-dim-mu-only-s} 
and~\eqref{eq-dim-nu-only-sig} in Assumptions~\ref{ass-doubl} imply
\begin{equation*}   
\Theta_{q,\la}(r) \le C_0 \sup_{0<\rho\le r} \rho^{\al+\sig/q-s/p}.\qedhere
\end{equation*}
\end{proof}

Theorem~\ref{thm-comp-for-q'<q-gen} and
Corollary~\ref{cor-comp-dim} apply in particular to $\Y=M^{\al,p}(X,\mu)$
and extend earlier results from  
Haj\l asz~\cite[Theorem~8.7]{Haj-ContMat}, dealing with $\nu=\mu$,
 $E=X$ and $\al\le1$
(the case $\al<1$ follows from~\cite{Haj-ContMat} 
by using the snowflaked metric $d(x,y)^\al$).
Note that the dimension condition \eqref{eq-dim-mu-only-s} in
Assumptions~\ref{ass-doubl} is necessary for such Sobolev embeddings 
when $X$ is bounded and uniformly perfect, 
by Alvarado--G\'orka--Haj\l asz~\cite[Corollary~24]{AlGoHaj}.
In Proposition~\ref{prop-Th<infty-necess}, we extend their
  result to the two-weighted situation.

Since certain types of Besov spaces embed into $M^{\al,p}$, by e.g.\ 
Gogatishvili--Koskela--Shanmugalingam~\cite[Lemma~6.1]{GoKoSh} and
Mal\'y~\cite[Lemma~3.16]{MalyBesov}, similar embeddings for Besov spaces
also follow from Theorem~\ref{thm-comp-for-q'<q-gen} and
Corollary~\ref{cor-comp-dim}.

The following examples illustrate why we only impose the doubling and dimension
conditions in Assumptions~\ref{ass-doubl} on balls centred in $E$.
The local assumption about the \p-Poincar\'e inequalities~\eqref{eq-def-abstr-PI}
and~\eqref{PI-ineq}  can be justified similarly.

\begin{example}  \label{eq-why-s-on-E}
Consider $X=\R^n$ with the weight 
\[
w(x_1,\ldots,x_n)=|x_1|^{-\theta}, \quad x_1\ne0 \text{ and } 0<\theta<1.
\]
Since $w$ is easily verified to be an $A_1$ weight (see~\cite[p.10]{HeKiMa}),
the measure $d\mu=w\,dx$ is doubling and supports 
a 1-Poincar\'e inequality (i.e.\ \eqref{eq-def-abstr-PI} with $p=\al=\la=1$)
for all balls $B\subset\R^n$.
It is easily verified that for $x=(x_1,x_2,\ldots,x_n)$ and $r>0$, 
the measure $\mu(B(x,r))$ is comparable to $r^{n-\theta}$ if $r\ge\tfrac12|x_1|$
and to $|x_1|^{-\theta}r^n$ otherwise.

This implies that if $E$ is contained in the hyperplane $\{x\in\R^n: x_1=0\}$
then the dimension condition   
\eqref{eq-dim-mu-only-s} for $\mu$ and balls centred on $E$
holds with $s=n-\theta<n$, 
while for general balls centred in $\R^n$ we must have $s\ge n$. 
Since the critical exponent $q=\sig p/(s-\al p)$ in Corollary~\ref{cor-comp-dim}
increases as $s$ decreases, 
allowing for the smaller local ``dimension'' $s=n-\theta$ gives better
embedding results than $s=n$. 
A similar argument applies e.g.\ when considering trace embeddings
with respect to the weight $w(x)=\dist(x,\bdry\Om)^{-\alpha}$ for
sufficiently regular domains $\Om\subset\R^n$, 
see Theorem~\ref{thm-mu-al-new} and Example~\ref{ex-trace-emb} below.
\end{example}

\begin{example}  \label{ex-Whitney-modif}
Let $E\subset\R^n$ be compact and consider a weight $0<w\in L^1\loc(\R^n)$
satisfying $w\equiv1$ on $E$ and $\vint_Q w\,dx=1$ for every cube $Q$
in the Whitney decomposition of $\R^n\setm E$. 
The weighted measure $d\mu(x)=w(x)\,dx$ then satisfies, for 
all balls centred in $E$, the same doubling and dimension conditions
in Assumptions~\ref{ass-doubl} as the Lebesgue measure. 
It can therefore be used in Theorem~\ref{thm-comp-for-q'<q-gen}
  and Corollary~\ref{cor-comp-dim}.

However, $\mu$ can fail any doubling conditions for balls
outside $E$ and the restriction $\mu|_E$ need not be doubling either.
Thus, Assumptions~\ref{ass-doubl} are substantially weaker than
similar global conditions or assumptions for the restricted measure.
\end{example}

\section{Newtonian  spaces and 
Theorem~\ref{thm-Mazya-cond-intro-new}}
\label{sect-Newt}

Another generalization of Sobolev spaces to
metric spaces are Newtonian spaces $\Np$.
Their definition is based on the following notion
of upper gradients, introduced in Heinonen--Koskela~\cite{HeKo98}:
A Borel function $g:X\to[0,\infty]$ 
is an \emph{upper gradient} of $u:X\to[-\infty,\infty]$
if for all rectifiable curves  $\gamma: [0,l_{\gamma}] \to X$,
\begin{equation} \label{ug-cond}
        |u(\gamma(0)) - u(\gamma(l_{\gamma}))| \le \int_{\gamma} g\,ds,
\end{equation}
where $ds$ is the arc length parametrization of $\ga$
and the left-hand side is interpreted as $\infty$ if at least
one of the terms therein is infinite.

If $u$ has an upper gradient in $L\loc^p(X,\mu)$, then
it has a $\mu$-a.e.\ unique \emph{minimal {\rm(}\p-weak\/{\rm)} upper gradient} 
$g_u \in L\loc^p(X,\mu)$, see Shan\-mu\-ga\-lin\-gam~\cite[Corollary~3.7]{Sh-harm}.

\begin{deff}  
The \emph{Newtonian space} on $X$ is
\[
\Np (X,\mu) = \biggl\{u:  \|u\|_{\Np(X,\mu)} := \biggl( \int_X |u|^p \, \dmu
                + \inf_g\int_X g^p \, \dmu \biggr)^{1/p}  <\infty \biggr\},
\]
where the infimum is taken over all (\p-weak) upper gradients $g$ of $u$.
The \emph{Dirichlet space} $\Dp(X,\mu)$ consists of all $u\in L^1\loc(X,\mu)$
having an upper gradient in $L^p(X,\mu)$, and
is equipped with the seminorm $\|g_u\|_{L^p(X,\mu)}$.
\end{deff}

Newtonian spaces were defined by Shan\-mu\-ga\-lin\-gam~\cite{Sh-rev}
and are in general larger than both $\Mp$ and $P^{1,p}_\tau$, see
Haj\l asz~\cite[Corollary~10.5]{Haj-ContMat}.
For example, if $X$ has no rectifiable curves then 
$N^{1,p}(X,\mu)=L^p(X,\mu)$,
while $M^{1,p}(X,\mu)$ is in general different from $L^p(X,\mu)$,
see Proposition~\ref{prop-q=p-M-al}.
The truncation property assumed in Theorems~\ref{thm-two-weight-PI-eps}\,\ref{it-trunc-prop} and 
\ref{thm-comp-for-q'<q-gen}\,\ref{it-Th-trunc} and Corollary~\ref{cor-comp-dim}\,\ref{it-dim-trunc}
is valid for $\Np$ and the minimal \p-weak upper gradients and makes it possible
to treat also the limiting exponent $q'=q$ and to
obtain the optimal sufficient and necessary conditions 
in Theorem~\ref{thm-Mazya-cond-intro-new}.

For various properties of Newtonian functions we refer to
Shanmugalingam~\cite{Sh-rev}, \cite{Sh-harm}, Bj\"orn--Bj\"orn~\cite{BBbook}
and Heinonen--Koskela--Shanmugalingam--Tyson~\cite{HKST}.
For \eqref{ug-cond} to make sense,
functions in $\Np(X,\mu)$ and $\Dp(X,\mu)$ are pointwise
defined.
Their equivalence classes are up to sets of zero \p-capacity, where the
(\emph{Sobolev}\/) \p-\emph{capacity} of a set $A \subset X$ is
\begin{equation*}
   \Cp (A) =\inf    \|u\|_{\Np(X,\mu)}^p,
\end{equation*}
with the infimum taken over all $u\in \Np (X,\mu)$ such that
$u\ge 1$ on $A$.

\begin{remark}[\emph{$\Np$ for domains in $\R^n$}]  
By~\cite{Sh-rev}, the space $\Np(\Om,dx)$ coincides for $p>1$ and open $\Om\subset\R^n$
with the usual Sobolev space $W^{1,p}(\Om)$, and 
\begin{equation} \label{eq-g=grad}
g_u=|\grad u| \quad \text{a.e.},
\end{equation}
where $\grad u$ is the distributional gradient.
However, since the equivalence classes in $\Np(\Om,dx)$
are taken with respect to sets of \p-capacity zero, 
$\Np(\Om,dx)$ 
consists only of the quasicontinuous representatives from $W^{1,p}(\Om)$,
considered e.g.\ in Evans--Gariepy~\cite[Chapter~4]{EvGa92}.  

Similarly, \cite[Appendix~A.1]{BBbook} shows that 
if $0<w\in L^1\loc(\Om)$ is a \p-admissible weight in the sense of 
Heinonen--Kilpel\"ainen--Martio~\cite{HeKiMa}, then
$\Np(\Om,\mu)$ corresponds to the weighted
(refined) Sobolev space studied in~\cite[Chapter~4]{HeKiMa}.
In this case, \eqref{eq-g=grad}
holds with $\grad u$ defined through the completion of $C^\infty$ as in
\cite[Section~1.9]{HeKiMa} (and equal to the distributional gradient
when $w^{1/(1-p)}\in L^1\loc(\Om)$).

For general weights, 
the relationship between $\Np$ and other definitions of weighted Sobolev spaces
is rather subtle,  
see Ambrosio--Pinamonti--Speight~\cite{AmPiSp} and Zhikov~\cite{Zhikov}
for counterexamples and sufficient integrability conditions on $w$.
\end{remark}

When $\mu$ is doubling and supports the \p-Poincar\'e
inequality~\eqref{PI-ineq} with $g_u$ on $X$, we have
for $p>1$ and up to suitable equivalence classes and equivalent norms, 
\begin{equation*}   
N^{1,p}(X,\mu) = M^{1,p}(X,\mu) = P^{1,p}_\tau(X,\mu) \cap L^p(X,\mu),
\end{equation*}
see~\cite[Theorem~11.3]{Haj-ContMat}
(together with \cite[Theorem~5.1]{noncomp} when $X$ is not complete).

\begin{remark}[\emph{Embeddings and Lebesgue points for $\Np$}]   
The question of $\ub$ in \eqref{eq-Leb-repr} and $\mu$-Lebesgue points in
Lemma~\ref{lem-est-Leb-pts}
is more delicate for $\Np$ than for $P^{\al,p}_\tau$ and $M^{\al,p}$,
since the equivalence classes in $\Np$ are up to sets of zero \p-capacity.
More precisely, 
$\ub$ might differ from $u$ more than on a set of zero \p-capacity, 
in which case it is not a representative of $u$ in $\Np(X,\mu)$, only
a $\mu$-a.e.-representative.
It is this function $\ub$ that realizes the embeddings into
$L^q(E,\nu)$ in Theorem~\ref{thm-Mazya-cond-intro-new}.

A sufficient condition for $\ub\in\Np(X,\mu)$ is that 
the doubling property for $\mu$
and the Poincar\'e inequality~\eqref{PI-ineq} with $g_u$ 
hold for all balls 
contained in some open neighbourhood of $E$.
Indeed, Proposition~4.8 in~\cite{noncomp} (for $p>1$) and
Proposition~3.7 in~\cite{panu} (for $p=1$) 
guarantee that $\ub=u$ on $E$ outside a set of zero \p-capacity.
\end{remark}

\begin{proof}[Proof of Theorem~\ref{thm-Mazya-cond-intro-new}]
For the sufficiency part, note that the minimal \p-weak upper gradients $g_u$
satisfy the truncation property.
More precisely, for all $l<k$, the functions $g_u\chi_{\{l<u<k\}}$ are 
\p-weak upper gradients of the truncations $\max\{l,\min\{u,k\} \}$
and hence the \p-Poincar\'e inequality \eqref{PI-ineq} 
(and thus \eqref{eq-def-abstr-PI} with $\al=1$) holds for these functions as well. 
Theorem~\ref{thm-comp-for-q'<q-gen}\,\ref{it-Th-trunc} 
thus applies with $G(u)= g_u$.
Since $\al=1$, the conditions $\Theta_{q,\la}(r)<\infty$ and $\Theta_{q,\la}(r)\to0$
become \eqref{eq-bdd-cond-intro} and \eqref{eq-Mazya-cond-intro-new}.

The necessity part follows from Propositions~\ref{prop-comp-imp-M-to0}
and~\ref{prop-Th<infty-necess}, since their conclusions imply
\eqref{eq-bdd-cond-intro} and
\eqref{eq-Mazya-cond-intro-new}, by the measure density condition~\eqref{eq-meas-density}.
\end{proof}

\section{Necessary conditions for compact embeddings}
\label{sect-optimality} 

Condition~\eqref{eq-Mazya-cond-intro-new}
can be compared to Theorem~8.8.3 in Maz\cprime ya~\cite{Mazya},
where the embedding
\(
W^{1,p}(\R^n) \longemb L^q(\R^n,\nu),
\)
was shown to be compact if and only if
\begin{equation}   \label{eq-Mazya-comp-cond}
\lim_{r\to0} \sup_{x\in\R^n} \frac{r \nu(B(x,r))^{1/q}}{|B(x,r)|^{1/p}} = 0
\quad \text{and} \quad 
\lim_{|x|\to\infty} \sup_{0<r\le1}
            \frac{r \nu(B(x,r))^{1/q}}{|B(x,r)|^{1/p}} = 0,
\end{equation}
where $q>p$ and $|B(x,r)|$ is the Lebesgue measure of $B(x,r)$.
Of course, for embeddings into $L^q(E,\nu)$ with bounded $E$,
the latter condition in \eqref{eq-Mazya-comp-cond} is superfluous.

We shall now show that similar conditions are
necessary for compactness and boundedness
also in metric spaces, under very mild assumptions 
on the measures.

\begin{prop}[Necessity of Maz\cprime ya's condition]
\label{prop-comp-imp-M-to0}
Assume that $\nu$ is nonatomic on $E$, i.e.\ $\nu(\{x\})=0$ for all $x\in E$.
If the embedding 
$\Y \embed L^q(E,\nu)$ is compact, where $\Y$ is $\Np(X,\mu)$, $\Dp(X,\mu)$,
$M^{1,p}(X,\mu)$, $P^{1,p}_1(X,\mu)$ or $P^{1,p}_{1,E}(X,\mu)$, then  for all $\la>1$,   
\begin{equation}   \label{eq-max-attained}
\sup_{x\in X}
     \frac{r \nu(B(x,r)\cap E)^{1/q}}{\mu(B(x,\la r))^{1/p}}
\to 0, \quad \text{as } r\to0,
\end{equation}
where the supremum is taken only over 
balls such that $0<\mu(B(x,\la r))<\infty$.
\end{prop}

\begin{proof} 
Let $\{r_j\}_{j=1}^\infty$ be an arbitrary sequence converging to zero
as $j\to\infty$.
For each $j=1,2,\ldots$\,, find a ball $B_j=B(x_j,r_j)$ such that
\[
\frac{r_j \nu(B_j\cap E)^{1/q}}{\mu(\lambda B_j)^{1/p}}
\ge \min \biggl\{1,\frac{{\cal S}(r_j)}{2}\biggr\},
\]
where ${\cal S}(r)$ denotes the supremum in \eqref{eq-max-attained}.
Consider the functions
\begin{equation}   \label{eq-def-uj}
u_j(x) = a_j \biggl( 1 - \frac{\dist(x,B_j)}{(\la-1)r_j} \biggr)_\limplus
\quad \text{with constants } 
a_j =\frac{(\la-1)r_j}{\mu(\la B_j)^{1/p}}.
\end{equation}
It is easily verified that the distance function $x\mapsto\dist(x,B_j)$ is
1-Lipschitz, and hence $g_j:=\chi_{\la B_j}/\mu(\la B_j)^{1/p}$
is a \p-weak upper gradient of $u_j$ and can be used in the
definition of $M^{1,p}(X,\mu)$ and the other spaces as well,
cf.\  \eqref{eq-P-subset-P,E} and \eqref{eq-M-subset-P}.
As both $u_j$ and $g_{j}$ vanish outside $\la B_j$, we have
\[
\int_X u_j^p\,d\mu \le a_j^p \mu(\la B_j) \le ((\la-1)r_j)^p
\quad \text{and} \quad
\int_X g_{j}^p\,d\mu \le 1,
\]
i.e.\ $\{u_j\}_{j=1}^\infty$  is bounded in $\Y$. 
By compactness of the embedding $\Y\embed L^q(E,\nu)$, 
there exists a subsequence (also denoted
$\{u_j\}_{j=1}^\infty$) which converges  in $L^q(E,\nu)$ to some $u\in L^q(E,\nu)$. 
Since $0<\mu(B)<\infty$ for all balls centred in $E$, 
\cite[Proposition~1.6]{BBbook} implies that $E$ is separable.
We therefore conclude from Lemma~\ref{lem-uj-to-0} below
that $u=0$ $\nu$-a.e.\ in $E$.
Hence, as $u_j=a_j$ in $B_j$, it follows from the choice of $B_j$ that
\[
\min \biggl\{ 1, \frac{{\cal S}(r_j)}{2} \biggr\}
\le \frac{r_j \nu(B_j\cap E)^{1/q}}{\mu(\lambda B_j)^{1/p}}
    \le \frac{1}{\la-1} \|u_j\|_{L^q(E,\nu)} \to0 \quad \text{as } 
j\to\infty,   
\]
which proves~\eqref{eq-max-attained}.
\end{proof}

\begin{lem}  \label{lem-uj-to-0}
Let $Y$ be a separable metric space equipped with a nonatomic 
Borel measure $\nu$. 
Assume that $u_j\in L^q(Y,\nu)$ are such that $\supp u_j\subset B(y_j,r_j)=:B_j$ for some
$y_j\in Y$ and $r_j\to 0$.
If $u_j\to u$ in $L^q(Y,\nu)$ then $u=0$ $\nu$-a.e.\ in $Y$.
\end{lem}

\begin{proof}
By passing to a subsequence, we can assume that $u_j\to u$
$\nu$-a.e.\ in $Y$.
Let $B=B(x,r)$ be an arbitrary ball. Then there are two possibilities:
\begin{enumerate}
\item \label{u=0-in-B}
There exists a subsequence $(B_{j_k})_{k=1}^\infty$ of $(B_j)_{j=1}^\infty$ 
such that $B_{j_k}\cap B=\emptyset$ for all $k=1,2,\ldots$\,.
Then $u=0$ $\nu$-a.e.\ in $B$ because of the $\nu$-a.e.\ convergence of $u_j$.
\item \label{u=0-outside-B}
There exists $k$ such that $B_{j}\cap B \ne\emptyset$ for all $j\ge k$,
and hence $\supp u_j$ lies within an $r_k$-neighbourhood of $B$. 
Letting $k\to\infty$ and hence $r_k\to0$, the $\nu$-a.e.\ convergence of 
$u_j$ then implies that $u=0$ $\nu$-a.e.\ in $Y\setm \itoverline{B}$.
\end{enumerate}
Now, we distinguish two cases:

If there is $x\in Y$ such that \ref{u=0-outside-B} 
holds for all $r>0$ then $u=0$ $\nu$-a.e.\ in $Y\setm \{x\}$ and hence
also $\nu$-a.e.\ in $Y$ by the assumption that $\nu$ is nonatomic.

Assume therefore that for every $x\in Y$, there exists $r_x>0$ such that
\ref{u=0-outside-B} fails for $r_x$, in which case \ref{u=0-in-B} must
hold for $B_x=B(x,r_x)$.
Since $Y$ is separable, and thus Lindel\"of (by e.g.\ 
\cite[Proposition~1.5]{BBbook}),
we can choose among the balls $B_x$, $x\in Y$, countably many balls
$B_{x_j}$, $j=1,2,\ldots$\,, so that  $Y\subset \bigcup_{j=1}^\infty B_{x_j}$.
As $u=0$ $\nu$-a.e.\ in each $B_{x_j}$, the claim follows by the
subadditivity of $\nu$.
\end{proof}

\begin{prop}  \label{prop-Th<infty-necess}
If the embedding 
$\Y \embed L^q(E,\nu)$ is bounded, where $\Y$ is $\Np(X,\mu)$, $\Dp(X,\mu)$,
$M^{1,p}(X,\mu)$, $P^{1,p}_1(X,\mu)$ or $P^{1,p}_{1,E}(X,\mu)$,  then for all  $\la>1$ and $r_0>0$,
\begin{equation*}   
\sup_{0<r\le r_0} \sup_{x\in X}
     \frac{r \nu(B(x,r)\cap E)^{1/q}}{\mu(B(x,\la r))^{1/p}} <\infty,
\end{equation*}
where the suprema are taken only over 
balls such that $0<\mu(B(x,\la r))<\infty$.

\end{prop}

\begin{proof}
For $z\in X$ and $0<r\le r_0$, let $B=B(z,r)$ and
\[
u(x) = a \biggl( 1 - \frac{\dist(x,B)}{(\la-1)r} \biggr)_\limplus
\quad 
\text{with } 
a =\frac{(\la-1)r}{\mu(\la B)^{1/p}}.
\]
As in the proof of Proposition~\ref{prop-comp-imp-M-to0}, $g= \chi_{\la B}/\mu(\la B)$
serves as a gradient in the above choices of $\Y$.
The boundedness of $\Y \embed L^q(E,\nu)$  yields  
\[
\frac{r \nu(B(z,r)\cap E)^{1/q}}{\mu(B(z,\la r))^{1/p}}
\le \frac{1}{\la-1} \|u\|_{L^q(E,\nu)} \le \frac{C}{\la-1} \|g\|_{L^p(X,\mu)}
\le \frac{C}{\la-1} <\infty.
\]
Taking supremum over all $z\in X$ and  $0<r\le r_0$ concludes the proof.
\end{proof}

\begin{prop}  \label{prop-tot-bdd}
Assume that $\mu$ satisfies both the doubling condition \eqref{eq-def-doubl-on-E} in
Assumptions~\ref{ass-doubl} and a measure density
condition as in~\eqref{eq-meas-density} for all balls centred in $E$ and of radius 
at most $r_0>0$.  
{\rm(}For example, let $\mu$ be doubling on $E=X$.{\rm)}
If the embedding $\Y \embed L^p(E,\mu)$
is compact, where $\Y$ is $\Np(X,\mu)$, $\Dp(X,\mu)$, $M^{1,p}(X,\mu)$ or
$P^{1,p}_{\tau,E}(X,\mu)$ {\rm(}with $\tau\ge1$ and $0<r\le r_0/\tau$ in
\eqref{def-PI-ineq-spc}{\rm)}, 
then $E$ is totally bounded.
\end{prop}

\begin{proof}
Assume that $E$ is not totally bounded. 
Then there exists $r\le r_0$ and
infinitely many balls $B_j=B(x_j,r) \subset X$, $j=1,2,\ldots$\,, such that 
$x_j\in E$ and $2B_j\cap2B_k=\emptyset$ whenever $j\ne k$.
As in the proof of Proposition~\ref{prop-comp-imp-M-to0},
the sequence $\{u_j\}_{j=1}^\infty$, defined by \eqref{eq-def-uj} with $\la=2$
and $r_j=r$, is bounded in $\Y$.
(Note that the doubling property \eqref{eq-def-doubl-on-E} for $\mu$ implies that 
$M^{1,p}(X,\mu) \subset P^{1,p}_{\tau,E}(X,\mu)$,
see~\eqref{eq-P-tau-subset-tau} and~\eqref{eq-M-subset-P}.)

By assumption, $\{u_j\}_{j=1}^\infty$ contains
a subsequence (also denoted $\{u_j\}_{j=1}^\infty$) which converges both in 
$L^p(E,\mu)$ and $\mu$-a.e.\ in $E$ to some $u\in L^p(E,\mu)$. 
Since for every $x\in X$ at most one of $u_j(x)$ is nonzero
by the choice of $B_j$, we conclude that
$u=0$ $\mu$-a.e.\ in $E$ and hence $\|u_j\|_{L^p(E,\mu)}\to 0$.
This contradicts the doubling property \eqref{eq-def-doubl-on-E} 
and the measure density~\eqref{eq-meas-density} of~$\mu$, since
\[
\|u_j\|_{L^p(E,\mu)} \ge a_j \mu(B_j\cap E)^{1/p} 
= \frac{r \mu(B_j\cap E)^{1/p}}{\mu(2B_j)^{1/p}} \ge Cr > 0.\qedhere
\]
\end{proof}

\section{Examples}
\label{sect-examples}

We now give concrete examples when our results can be applied.
Our aim is mainly to demonstrate their flexibility.

\begin{example}[Lower-dimensional target measures]
\label{ex-d-Hausdorff}
A suitable candidate for $\nu$ is the $d$-dimensional Hausdorff measure $\La_d$ 
on a lower-dimensional Ahlfors regular $d$-set, as in 
Jonsson--Wallin \cite[Sections~2.1.1 and~2.1.2]{JonWal}.
More precisely, assume that the dimension condition~\eqref{eq-d-set} 
holds for $E\subset F\subset X$.
Assumptions~\ref{ass-doubl} then hold for $\nu=\La_d|_F$ with
$\de=\sig=d$ and Corollary~\ref{cor-comp-dim}   
yields compactness of embeddings $\Y\embed L^q(E,\La_d)$
when $\Y \subset P^{\al,p}_{\tau,E}(X,\mu)$, $\tau\ge1$, $q(s-\al p) < dp$, 
and $\mu$ satisfies the local doubling and dimension conditions~\eqref{eq-def-doubl-on-E}
and \eqref{eq-dim-mu-only-s}. 

A traditional approach to such embeddings is to
first use trace theorems into Besov spaces on $F$ 
and then embeddings of these Besov spaces into $L^q$ spaces,
followed by the compactness results from 
Haj\l asz--Koskela~\cite[Theorem~5]{HaKo-imbed} or
Haj\l asz--Liu~\cite{hajzu}.
See \cite[Theorem~VII.1 and Proposition~VIII.6]{JonWal} 
together with Peetre~\cite[Th\'eor\`eme~8.1]{PeetreEsp} for such results in
$\R^n$, and 
L.~Mal\'y~\cite[Corollary~3.18 and Proposition~4.16]{MalyBesov} on metric spaces. 

Our approach is direct and avoids the use of Besov spaces. 
We impose assumptions only on small balls centred in $E$,
which can give better exponents, as seen in Examples~\ref{eq-why-s-on-E}
and~\ref{ex-Whitney-modif}.
Our condition~\eqref{eq-dim-mu-only-s} for the exponent~$s$
is in general less restrictive than the one in \cite{MalyBesov}.
We use the Hausdorff measure $\La_d$ satisfying~\eqref{eq-d-set}, 
rather than the
codimensional bounds in \cite[Proposition~4.16]{MalyBesov}.
However, a codimensional measure, introduced by J.~Mal\'y~\cite{Maly}, could also
be used instead of $\La_d$.
\end{example}

Another application of our results is to replace $X$ and $E$ by a
sufficiently regular domain, seen as a metric spaces in its own right,
and equipped with weights given by the distance to the boundary.
An open set $\Om\subset X$ is a \emph{uniform domain},
if there is a constant $A\ge1$
such that for every pair $x, y\in\Om$
there is a curve $\gamma$ in $\Om$ connecting $x$ and $y$, so that its length
is at most $Ad(x,y)$ and for all $z\in\gamma$,
\begin{equation*}    
\dist(z,X\setm\Om) \ge A^{-1} \min\{\ell_{xz},\ell_{yz}\},
\end{equation*}
where $\ell_{xz}$ and $\ell_{yz}$ are the lengths of the subcurves 
of $\gamma$ connecting $z$ to $x$ and $y$, respectively.
Uniform domains were introduced by Martio--Sarvas~\cite{MarSar}.
Typical examples include convex sets 
and bounded Lipschitz domains in $\R^n$, see Aikawa~\cite[p.~120]{Aik}
and Maz$'$ya~\cite[Sections~1.1.8--1.1.11]{MazSobBook}.
There are also many examples of fractal nature, such as the interior of the
von Koch snowflake.
By considering uniform domains as metric spaces in their own right, 
we obtain the following embedding result.
The assumptions on $\mu$ and $\nu$ guarantee that results from~\cite{JMAA}
can be used to obtain suitable estimates and Poincar\'e inequalities 
for the measures in~\eqref{eq-def-mu-al-nu-be}.

\begin{thm}  [Uniform domains with distance weights]
\label{thm-mu-al-new}
Assume the following for a bounded uniform domain $\Om\subset X$\/{\rm:}
\begin{enumerate}
\item \label{it-Ass-D-E}
The local doubling and dimension
conditions~\eqref{eq-def-doubl-on-E}--\eqref{eq-dim-nu-only-sig} 
in Assumptions~\ref{ass-doubl} 
hold for $\mu$ and $\nu$ with exponents $s$ and $\sig>s-p$, for all
balls centred in $\Om$.
\item \label{it-ass-p-PI-gu}
The domain measure $\mu$ supports the Poincar\'e inequality~\eqref{PI-ineq} with
$g_u$ and $\la\ge1$ for all balls $B\subset\Om$.
\end{enumerate}
Consider the measures
\begin{equation}   \label{eq-def-mu-al-nu-be}
d\mu_\al := d(\cdot)^\al\,d\mu|_\Om 
\quad  \text{and} \quad
d\nu_\be := d(\cdot)^\be\,d\nu|_\Om \quad \text{on }\Om,
\end{equation}
where $d(x):=\dist(x,X\setm\Om)$ and $\al,\be\in\R$. 
Then there exist constants $\al_0,\be_0<0$ such that if $\al>\al_0$
and $\be>\be_0$, then the embeddings
\begin{equation*}
\Dp(\Om,\mu_\al) \cap L^1(\Om,\nu_\al) \longemb L^q(\Om,\nu_\al)
\quad \text{and} \quad
\Np(\Om,\mu_\al) \longemb L^q(\Om,\nu_\al)
\end{equation*}
are compact whenever one of the following holds\/{\em:}
\begin{enumerate}
\renewcommand{\theenumi}{\textup{(\roman{enumi})}}%
\renewcommand{\labelenumi}{\theenumi}
\setcounter{enumi}{0}
\item \label{it-large-q} $q>p$ and
\begin{equation}  \label{eq-bdd-with-al-be}
q(s-p) < \sig p + \min\{\be p - \al q,0\},
\end{equation}
\item \label{it-small-q}
$q\le p$ and $s-p<\sig+\be-\al$.
\end{enumerate}
If $q>p$ and the inequality in \eqref{eq-bdd-with-al-be} is 
nonstrict, then the above embeddings are bounded.
Moreover, if $\nu=\mu$, then assumption~\eqref{eq-dim-nu-only-sig}
in~\ref{it-Ass-D-E} is not needed and
the embedding $\Np(\Om,\mu_\al)\embed L^q(\Om,\mu_\be)$
is compact {\rm(}resp.\ bounded\/{\rm)} whenever the above assumptions
{\rm(}resp.\ a nonstrict analogue of \eqref{eq-bdd-with-al-be}\/{\rm)}
hold with $\sig$ replaced by~$s$.
\end{thm}

\begin{remark}    
Compactness of the embedding $M^{\theta,p}(\Om,\mu_\al)\embed L^q(\Om,\mu_\be)$
for $\theta>0$ can be shown similarly when $\Om$, rather than being uniform,
satisfies the \emph{corkscrew condition} 
(i.e.\ there is $c>0$ such that for all $x\in\clOm$ and
$0<r\le\diam\Om$, the set $B(x,r)\cap\Om$ contains a ball of radius $cr$) and $s-p$ 
in assumptions~\ref{it-Ass-D-E}, \eqref{eq-bdd-with-al-be}
and~\ref{it-small-q} of Theorem~\ref{thm-mu-al-new}
is replaced by $s-\theta p$.
Indeed, the \p-Poincar\'e inequality for $g_u$ in \ref{it-ass-p-PI-gu}
can be replaced by~\eqref{eq-p,p-PI-M}.
\end{remark}

\begin{remark}    
The proof presented below shows that $\mu_\al$ and $\nu_\be$ satisfy
\eqref{eq-dim-mu-only-s} and \eqref{eq-dim-nu-only-sig} 
with $s$ and $\sig$ replaced by 
\[
s_\al=\max\{s,s+\al\} \quad \text{and} \quad 
\sig_\be=\min\{\sig,\max\{\sig+\be,0\}\}, \quad \text{respectively}.
\]
However, a direct use of Corollary~\ref{cor-comp-dim} with these
exponents would instead of $\sig>s-p$   and the assumptions
\ref{it-large-q} and \ref{it-small-q} in Theorem~\ref{thm-mu-al-new}
require that 
\[
\sig_\be > s_\al-p    
\quad \text{and} \quad   q(s_\al-p) < \sig_\be p.
\]
This would give a less general result than our direct proof (e.g.\ 
when $\al,\be>0$).
\end{remark}

\begin{proof} [Proof of Theorem~\ref{thm-mu-al-new}] 
We shall use Theorem~\ref{thm-comp-for-q'<q-gen}
with $X$ and $E$  replaced by $\Om$.
The doubling property of $\mu$ (and thus of $X$)
implies that the bounded set $\Om$ is totally bounded,
see Heinonen~\cite[Section~10.13]{Heinonen}.
Theorem~4.4 in \cite{JMAA} and its proof ensure that for some $\al_0<0$ and all $\al>\al_0$, 
the measure $\mu_\al$ is doubling and
supports the \p-Poincar\'e inequality~\eqref{PI-ineq} for $g_u$ 
on $\Om$ as the underlying space, with dilation $\la=3A$, where $A$ is the
uniformity constant for $\Om$.
Hence, $\Dp(\Om,\mu_\al)\subset P^{1,p}_{3A,\Om}(\Om,\mu_\al)$.

Since $\Om$ satisfies the corkscrew condition, by \cite[Lemma~4.2
and Definition~2.4]{JMAA}, 
we conclude from \cite[Theorem~2.8]{JMAA} that $\nu_\be$ is doubling on
$\Om$ for all $\be>\be_0$ and some $\be_0<0$.
In particular, \eqref{eq-def-doubl-on-E} is satisfied for $\mu_\al$ and $\nu_\be$.
As $\Om$ is connected, also \eqref{eq-dim-est-sig-E} holds for $\nu_\be$, 
see Remark~\ref{rem-doubl-imp-dim}.
To be able to control the local Poincar\'e constant 
\[
\Theta_{q,\la}(r)=\sup_{0<\rho\le r}
\sup_{x\in \Om} \frac{\rho \nu_\be(B(x,\rho))^{1/q}}
{\mu_\al(B(x,\la \rho))^{1/p}},
\]
we need to estimate $\mu_\al(B(x,\rho))$ 
and $\nu_\be(B(x,\rho))$ for $x\in\Om$ and $0<\rho\le r_0$.

First, \cite[Theorem~2.8 (with $\nu$ replaced by $\mu_\al$) 
and Lemma~4.2]{JMAA} show that
\begin{equation*}   
\mu_\al(B(x,\rho)) \ge C \rho^\al \mu(B(x,\rho)\cap\Om) 
\ge C \rho^\al \mu(B(x,\rho))
\quad \text{for } \rho\ge\tfrac12d(x).
\end{equation*}
Similarly, when $\be\le0$, \cite[Theorem~2.8]{JMAA} with $\mu$ replaced by $\nu_\be$ 
implies that 
\[
\nu_\be(B(x,\rho)) \le C \rho^\be \nu(B(x,\rho))
\quad \text{for } \rho\ge\tfrac12d(x),
\]
while the inequality is immediate when $\be>0$.
If $0<\rho\le\tfrac12d(x)$ then $d(y)$ is comparable to $d(x)$ on 
$B(x,\rho)$, which immediately yields that $\mu_\al(B(x,\rho))$ and $\nu_\be(B(x,\rho))$
are  comparable to $d(x)^\al\mu(B(x,\rho))$ and $d(x)^\be\nu(B(x,\rho))$, respectively.

Using this, we now have that
\begin{equation}    \label{eq-s-al-de(x)}
\frac{\nu_\be(B(x,\rho))^{1/q}}{\mu_\al(B(x,\rho))^{1/p}} \le
\begin{cases} 
\displaystyle \frac{C \nu(B(x,\rho))^{1/q}}{\mu(B(x,\rho))^{1/p}} 
             \, d(x)^{\be/q-\al/p},
         & \text{if } 0<\rho\le\tfrac12d(x), \\
\displaystyle \frac{C \nu(B(x,\rho))^{1/q}}{\mu(B(x,\rho))^{1/p}} 
             \, \rho^{\be/q-\al/p},
         & \text{if } \rho\ge\tfrac12d(x).
\end{cases}   
\end{equation}
Since 
\[
d(x)^{\be/q-\al/p} \le C \rho^{\min\{\be/q-\al/p,0\}},
\quad \text{when }0<\rho\le\tfrac12d(x)\le \tfrac12\diam\Om<\infty,
\]
we therefore conclude, using \eqref{eq-dim-mu-only-s}
and \eqref{eq-dim-nu-only-sig}, that in both cases, 
\[
\Theta_{q,\la}(r) \le \sup_{0<\rho\le r} \sup_{x\in\Om}
    \frac{\rho \nu_\be(B(x,\rho))^{1/q}}{\mu_\al(B(x,\rho))^{1/p}}
\le C_0 \sup_{0<\rho\le r} \rho^{1+\sig/q-s/p+\min\{\be/q-\al/p,0\}},
\]
where $C_0$ depends on $r_0$ and other fixed parameters, but not on $r$.
It follows that $\Theta_{q,\la}(r)\to0$ whenever the exponent 
\begin{equation}   \label{eq-exponent-ge0}
1 +\frac{\sig}{q} -\frac{s}{p}
   +\min \biggl\{ \frac{\be}{q} -\frac{\al}{p},0 \biggr\} >0,
\end{equation}
and that $\Theta_{q,\la}(r)<\infty$ 
if the inequality in \eqref{eq-exponent-ge0} is nonstrict. 
It is easily verified that \eqref{eq-exponent-ge0} follows from~\eqref{eq-bdd-with-al-be},  
and hence an application of Theorem~\ref{thm-comp-for-q'<q-gen}
concludes the proof of the general case $\nu\ne\mu$ when $q>p$.

When $q\le p$, the assumptions $\sig>s-p$ and $s-p<\sig+\be-\al$ 
imply that there is always $\qb>p$ satisfying~\eqref{eq-bdd-with-al-be}
(instead of $q$).
Thus, an application of the already proved part,
with $q>p$ replaced by $\qb>p$, provides compact embeddings
into $L^{\qb}$ and thus also into $L^q$.

If $\nu=\mu$, then replacing $\nu$ by $\mu$ in \eqref{eq-s-al-de(x)}
gives that
\[
\Theta_{q,\la}(r) \le C \sup_{0<\rho\le r} \sup_{x\in \Om} 
       \mu(B(x,\rho))^{1/q-1/p} \rho^{1+\min\{\be/q-\al/p,0\}}.
\]
The requirement~\eqref{eq-exponent-ge0} is therefore replaced by 
\[
1 +\frac{s}{q} -\frac{s}{p}
   +\min \biggl\{ \frac{\be}{q} -\frac{\al}{p},0 \biggr\}>0,
\]
which holds when $\sig$ in \eqref{eq-bdd-with-al-be} is replaced by $s$. 
\end{proof}

\begin{proof}[Proof of Proposition~\ref{prop-bdry}]
Note that $M^{\al,p}(\Om,\mu)=M^{\al,p}(\clOm,\mu|_\Om)$, with $\mu|_{\Om}$ extended by zero to $\bdry\Om$.
Remark~\ref{rem-Leb-pts} and Lemma~\ref{lem-est-Leb-pts} then imply that
the limit in~\eqref{eq-extend-limsup} exists $\La_d$-a.e.\ in $E$.
Corollary~\ref{cor-comp-dim}\,\ref{it-dim-comp}, applied to $\Y=M^{\al,p}(\clOm,\mu|_\Om)$ and
$\nu=\La_d|_{F}$, proves the statement for $M^{\al,p}$.

For $\Np$, it follows 
from~Aikawa--Shan\-mu\-ga\-lin\-gam~\cite[Proposition~7.1]{AiSh} that 
the measure $\mu|_{\Om}$, extended by zero to $\bdry\Om$,
is doubling and supports the \p-Poincar\'e inequality~\eqref{PI-ineq}
for $g_u$ on $\clOm$ as the underlying metric space. 
Moreover, $\Np(\Om,\mu)=\Np(\clOm,\mu|_\Om)$
with extensions given by \eqref{eq-extend-limsup} and with 
the same norm, see
Heinonen--Koskela--Shanmugalingam--Tyson~\cite[Lemma~8.2.3]{HKST}
together with Theorem~4.1 and Remark~4.2 in
Bj\"orn--Bj\"orn~\cite{noncomp}
(when $X$ is not complete).
Applying \ref{it-dim-comp} and~\ref{it-dim-trunc} in
Corollary~\ref{cor-comp-dim}  to $\Y=\Np(\clOm,\mu|_{\Om})$ concludes the proof.
\end{proof}

\begin{remark}     \label{rem-when-mu|Om}
Theorems~2.8 and~4.4 in Bj\"orn--Shanmugalingam~\cite{JMAA}
imply that the assumptions on $\mu|_\Om$ in
Proposition~\ref{prop-bdry} are satisfied e.g.\ if
\begin{enumerate}
\item[$\bullet$] $\Om$ satisfies the
  corkscrew condition and $\mu$ is doubling on $X$ (for $M^{\al,p}$).
\item[$\bullet$] $\Om$
is a uniform domain and $\mu$ is doubling and supports the 
\p-Poincar\'e inequality~\eqref{PI-ineq} for $g_u$ on $X$ (for
embeddings from $\Np$).
\end{enumerate}
\end{remark}

\begin{example}[\emph{Trace embeddings for Lipschitz domains}]  \label{ex-trace-emb}   
Since bounded Lipschitz domains $\Om\subset\R^n$ are uniform,  
Proposition~\ref{prop-bdry} and Remark~\ref{rem-when-mu|Om} imply
that the trace embedding
\[
W^{1,p}(\Om) \longemb L^q(\bdry\Om,\La_{n-1})
\]
is bounded whenever $q(n-p)\le p(n-1)$ and compact if $q(n-p)<p(n-1)$.
This recovers a classical result, see 
Section~6.10.5 in 
Kufner--John--Fu\v{c}\'\i k~\cite{kjf}.
\end{example}

The following example seems new in the weighted setting. 
A similar case with $\al<0$ was recently considered in Lindquist--Shanmugalingam~\cite{LinShan}.

\begin{example}[\emph{von Koch snowflake with a weight}]   
\label{ex-weighted-von-Koch}
Let $\Om$ be the bounded domain in $\R^2$ whose boundary is the
von Koch snowflake curve of  Hausdorff dimension $d=\log4/{\log3}$.
Then $\Om$ is a uniform domain.
Consider the weighted measure on $\Om$,
\[
d\mu_\al(x)=\dist(x,\bdry\Om)^\al\,dx, \quad  0\le\al<d+p-2.
\]
Theorems~2.8 and~4.4 in~\cite{JMAA} and
Proposition~\ref{prop-bdry} then show that the trace embedding
\[
W^{1,p}(\Om,\mu_\al) \longemb L^q(\bdry\Om,\La_d)
\]
is bounded whenever $q(2+\al-p)\le dp$ and compact if $q(2+\al-p)<dp$.
Other domains with fractal boundaries can be treated similarly.
\end{example}

\end{document}